\documentclass[12pt]{amsart}
\usepackage{amsmath,amssymb,amsfonts,amsthm,amsopn}
\usepackage{graphics}
\usepackage{mathrsfs}
\usepackage{graphicx,tikz}
\usepackage[all]{xy}
\usepackage{amsmath}
\usepackage{multirow, longtable, makecell, caption, array}
\usepackage{amssymb}
\usepackage{mathtools}
\usepackage{pb-diagram}
\usepackage{cellspace} %
  \def\Spnr{Sp(d,\R)}
  \def\Gltwonr{GL(2d,\R)}

\newcommand{\ft}{Fourier transform}

\newtheorem{tm}{Theorem}[section]
\newtheorem{lemma}[tm]{Lemma}

\newtheorem{theorem}{Theorem}[section]
\newtheorem{corollary}[theorem]{Corollary}
\newtheorem{definition}[theorem]{Definition}
\newtheorem{example}[theorem]{Example}

\newtheorem{proposition}[theorem]{Proposition}%%
\newtheorem{remark}[theorem]{Remark}

\newcommand{\beqa}{\begin{eqnarray*}}
\newcommand{\eeqa}{\end{eqnarray*}}

\newcommand{\field}[1]{\mathbb{#1}}
\newcommand{\bR}{\field{R}}        %  real numbers
\newcommand{\bZ}{\field{Z}}        %  whole numbers
\def\la{\lambda}

 \def\cF{\mathcal{F}}              % Calligraphic Letters
 \def\cS{\mathcal{S}}
 \def\cD{\mathcal{D}}
 
 \def\cB{\mathcal{B}}
 
 \def\cG{\mathcal{G}}
 \def\cM{\mathcal{M}}

 \def\cA{\mathcal{A}}
 
 \def\cI{\mathcal{I}}

\def\a{\aleph}
\def\hf{\hat{f}}

\def\rd{\bR^d}

\def\rdd{{\bR^{2d}}}

\def\lrd{L^2(\rd)}

\def\zd{\bZ^d}

\def\intrd{\int_{\rd}}
\def\intrdd{\int_{\rdd}}

\def\R{\right)}

\def\<{\left<}
\def\>{\right>}

\def\mv1{M_v^1}

\def\mpq{M^{p,q}}

\def\phas{(x,\xi )}

\def\o{\xi}
\def\a{\alpha}

\def\R{\mathbb{R}}
\def\Ren{\mathbb{R}^d}
\def\Renn{\mathbb{R}^{2d}}

\def\sch{\mathcal{S}}

\def\Fur{\mathcal{F}}

\def\f{\varphi}

\def\Opw{Op_{w}}
\def\Sn2{S_{2}(L^{2}(\Ren))}
\def\S1{S_{1}(L^{2}(\Ren))}
\def\sig00{\sigma_{0,0}}
\def\la{\langle}
\def\ra{\rangle}

\def\spdr{{\mathfrak {sp}}(d,\R)}

\newcommand{\A}{\mathcal{A}}

\begin{document}

\title[Characterization of modulation spaces]{Characterization of modulation spaces by symplectic representations and applications to Schr\"{o}dinger equations}
%----------Author 1
%\title{Symplectic time-frequency representations and characterization of modulation spaces}

\author{Elena Cordero and Luigi Rodino}
\address{Department of Mathematics,  University of Torino, Italy}
\address{Dipartimento di Matematica,  University of Torino, Italy}
%\curraddr{}
\email{elena.cordero@unito.it}
\email{luigi.rodino@unito.it}
%\thanks{This work was completed with the support of our
%----------Author 2

%\thanks{{\bf Acknowledgments.} The authors thank Maurice de Gosson and Luigi Rodino  for very helpful  discussions on this
%topic.}
\date{}
%\dedicatory{To my parents}
%%% ----------------------------------------------------------------------

\begin{abstract}
In the last twenty years modulation spaces, introduced  by H. G. Feichtinger in 1983, have been successfully addressed to the study of  signal analysis, PDE's, pseudodifferential operators, quantum mechanics, by  hundreds of contributions. In 2011 M. de Gosson  showed that the time-frequency representation {Short-time Fourier Transform} (STFT), which is the tool to define modulation spaces, can be  replaced by   the {Wigner distribution}. This idea was further generalized to $\tau$-Wigner representations in \cite{CR2021}. 

In this paper time-frequency representations are viewed as images of symplectic matrices via metaplectic operators. This new perspective highlights that the protagonists of time-frequency analysis are metaplectic operators and symplectic matrices $\cA \in Sp(2d,\bR)$. We find conditions on $\cA$ for which the related symplectic time-frequency representation $W_\cA$ can replace the STFT and give equivalent norms for weighted modulation spaces. In particular, we study the case of covariant matrices $\cA$, i.e., their corresponding $W_\cA$ are members of the Cohen class. 

Finally, we show that symplectic time-frequency representations $W_\cA$ can be efficiently  employed in the study of  Schr\"{o}dinger equations.  In fact, modulation spaces and $W_\cA$ representations are the frame for a new definition of wave front set, providing a sharp result for propagation of micro-singularities in the case of the quadratic Hamiltonians. This new approach  may have further applications in quantum mechanics and PDE's.
\end{abstract}

%%% ----------------------------------------------------------------------
\maketitle
%%% ----------------------------------------------------------------------

\section{Introduction}
Modulation spaces were originally introduced in 1983  by H. G. Feichtinger in the pioneering work \cite{feichtinger-modulation}. During the last twenty years hundreds of contributions have been written on the topic, showing that they are appropriate spaces  for a variety of fields, such as signal analysis, PDE's, pseudodifferential operators, quantum mechanics (a short non-exhaustive list of books and papers is \cite{KB2020,bogetal,cn met rep,Elena-book,Birkbis,grochenig,kasso07,RSTT2011,Sjostrand1,Toftweight2004,Wangbook2011,Wong}). 
The key-tool  for their definition is given by the time-frequency representation short-time Fourier transform (STFT) of a tempered
distributions $f\in\cS'(\rd)$  with respect to the Schwartz window function $g \in \cS(\rd)$, defined as
\begin{equation}\label{STFTdef}
V_gf\phas=\int_{\Ren}
f(y)\, {\overline {g(y-x)}} \, e^{-2\pi iy \o }\,dy,\quad (x,\xi)\in \rdd.
\end{equation}
 Given  indices $0<p,q\leq\infty$, the {\it
	modulation space} $M^{p,q}(\Ren)$ consists of all tempered
distributions $f\in\sch'(\Ren)$ such that $$V_gf\in L^{p,q}(\Renn )$$
(mixed-norm space) with $\|f\|_{M^{p,q}}\asymp \|V_gf\|_{L^{p,q}(\Renn )}$. For $p=q$ the notation $M^{p,p}(\rd)$ is shortened to $M^p(\rd)$ and we write $f\in M^p_{v_s}(\rd)$ if $V_gf\in L^p_{v_s}(\rdd)$ with the weight  $v_s(x,\xi):=(1+|(x,\xi)|^2)^{s/2}$.  For the main properties of these spaces, including the weighted versions,  we refer to Section $2$ below. 

In the realm of time-frequency representations another protagonist is given by the (cross-)Wigner distribution, introduced by Wigner in 1932 \cite{Wigner32} in Quantum Mechanics and, later, applied to many different environments such as PDE's and signal analysis. Namely, given a window function $g \in \cS(\rd)$, a tempered distribution $f$,
the (cross-)Wigner distribution $W(f,g)$ is given by
\begin{equation}\label{CWD}
W(f,g)\phas=\intrd f(x+\frac t2)\overline{g(x-\frac t2)}e^{-2\pi i t\o}\,dt,\quad (x,\xi)\in\rdd.
\end{equation} 
If $f=g$ we simply write $Wf=W(f,f)$ and call $Wf$ the Wigner distribution of $f$.

In 2011 M. de Gosson \cite{Birkbis} proved that in the definition of modulation spaces the STFT  could be  replaced by the cross-Wigner distribution. Hence \begin{equation}\label{2bis}
\|f\|_{M^{p,q}}\asymp \|W(f,g)\|_{L^{p,q}(\Renn)}.
\end{equation}  In our previous work \cite{CR2021} this idea was further generalized to $\tau$-Wigner representations $W_{\tau}(f,g)$, with $f,g$ as above, 
\begin{equation}
\label{tau-Wigner distribution}
W_{\tau}(f,g)(x,\xi) = \intrd e^{-2\pi it\xi}f(x+\tau t)\overline{g(x-(1-\tau)t)}dt,\quad \tau\in\bR
\end{equation} 
(for $f=g$ we obtain the $\tau$-Wigner distribution $W_\tau f:=	W_{\tau}(f,f)$;  
for $\tau=1/2$ we recapture the Wigner case). In fact, we showed that 
\begin{equation}\label{3bis}
\|f\|_{M^{p,q}}\asymp \|W_{\tau}(f,g)\|_{L^{p,q}(\Renn )},
\end{equation}
 for  $\tau\in \bR\setminus\{0,1\}$, whereas for $\tau=0$ or $\tau=1$, so-called Rihaczek distributions,  the previous characterization does not hold. The key observation was to interpret  the time-frequency representations above as images of symplectic matrices by metaplectic operators (defined as in the textbooks \cite{folland,Birkbis}).  
In fact, for any of them  we can find a symplectic matrix $\cA\in Sp(2d,\bR)$ such that the metaplectic operator $\mu(\cA)$ applied to $(f\otimes\bar{g})(x,\xi):=f(x)\bar{g}(\xi)$ coincides with it (for a suitable choice of the phase factor in the definition of $\mu(\cA)$).  For example, consider the symplectic matrix $\cA={\bf A}_{\tau}$, with 
\begin{equation}\label{Aw-tau} {\bf A}_{\tau}=\left(\begin{array}{cccc}
(1-\tau)I_{d\times d} &\tau I_{d\times d}&0_{d\times d}&0_{d\times d}\\
0_{d\times d}&0_{d\times d}&\tau I_{d\times d} &-(1-\tau)I_{d\times d}\\
0_{d\times d}&0_{d\times d}& I_{d\times d} &I_{d\times d}\\
-I_{d\times d}&I_{d\times d}& 0_{d\times d} &0_{d\times d}\\
\end{array}\right)\in Sp(2d,\bR),
\end{equation}
then 
$$\mu({\bf A}_{\tau})(f\otimes \bar{ g})=W_\tau(f,g),\quad \tau\in\bR.$$
Similarly, for $\cA= {\bf A_{ST}}$, where
 \begin{equation}\label{A-FT} {\bf A_{ST}}=\left(\begin{array}{cccc}
I_{d\times d} &- I_{d\times d}&0_{d\times d}&0_{d\times d}\\
0_{d\times d}&0_{d\times d}& I_{d\times d} &I_{d\times d}\\
0_{d\times d}&0_{d\times d}& 0_{d\times d} &-I_{d\times d}\\
-I_{d\times d}&0_{d\times d}& 0_{d\times d} &0_{d\times d}\\
\end{array}\right),
\end{equation}
we recapture the STFT:
$$\mu({\bf A_{ST}})(f\otimes \bar{ g})=V_gf.$$

This suggests a change of perspective: time-frequency representations can be viewed as images of metaplectic operators. Hence symplectic matrices and metaplectic operators may become the real protagonists in the framework of time-frequency analysis.

In this paper we show that symplectic matrices $\cA\in Sp(2d,\bR)$  are successfully employed to both recapture and find new time-frequency representations that we call {\bf $\cA$-Wigner distributions}:  $$W_\cA(f,g)=\mu(\cA)(f\otimes\bar{ g}).$$ 
For $f=g$ we simply write $W_{\cA}f:=W_\cA(f,f).$
%In particular, we find conditions on such matrices which provide equivalent norms for modulation spaces, that is,  for a fixed non-zero window function $g\in\cS(\rd)$,
%$$\|f\|_{M^{p,q}}\asymp \|W_\cA(f,g)\|_{L^{p,q}(\Renn)}, \quad 0<p,q\leq\infty$$ 
%(for weighted modulation spaces, see the characterization of Theorem \ref{caso-A-Wigner}  below).
The definition of the metaplectic operator $\mu(\cA)$ depends on the choice of a multiplicative  phase factor, which we omit for simplicity. \par
The properties of $\mu(\cA)$  are similar to those of the Wigner distribution, concerning in particular continuity on $L^2(\rd)$ (Proposition \ref{def bilA triple}), fundamental identity  for $W_{\cA}\widehat{f}$ (Proposition \ref{Prop2.7}) and Moyal identity (Proposition \ref{Moyal}).
Moreover, by using boundedness results for metaplectic operators on modulation spaces (Theorem \ref{Teorema0}, Corollary \ref{Cor3.14}) we may easily deduce the estimates
\begin{equation}\label{5A}
\|W_{\cA}(f,g)\|_{M^{p}_{v_{s}}}\lesssim \|f\|_{M^{p}}\|g\|_{M^{p}_{v_{s}}}+\|g\|_{M^{p}}\|f\|_{M^{p}_{v_{s}}}.
\end{equation}
and under the assumption $0<p\leq 2$ (Theorem \ref{Teorema2})
\begin{equation}\label{5B}
f\in M^{p}_{v_{s}}(\rd)\Leftrightarrow W_\cA f \in M^{p}_{v_{s}}(\rdd),
\end{equation}
which extends several results in literature, see \cite{Elena-book} and reference therein. \par
More challenging issue is to discuss the equivalence of norms for modulation spaces, that is, for a fixed non-zero window function $g\in\cS(\rd)$, 
\begin{equation}\label{5C}
\|f\|_{M^{p,q}}\asymp \|W_\cA(f,g)\|_{L^{p,q}}, \quad 0<p,q\leq\infty,
\end{equation}
in particular for $p=q$, allowing the presence of weights $v_s$:
\begin{equation}\label{5D}
\|f\|_{M^{p}_{v_s}}\asymp \|W_\cA(f,g)\|_{L^{p}_{v_s}}, \quad 0<p\leq\infty.
\end{equation}
Namely, we would like to extend in our context the characterizations of modulation spaces \eqref{2bis}, \eqref{3bis}.\par
In this perspective it is clear that we have to limit attention to subclasses of $Sp(2d,\bR)$. As a first attempt, it is natural to consider the covariant matrices $\cA$:
	\begin{equation*}
W_\cA(\pi(z)f,\pi(z)g)= T_z W_\cA (f,g), \quad  f,g\in\cS(\rd),\quad  z\in\rdd;
\end{equation*}
here for  $z=(z_1,z_2)$,  the operator $\pi(z)=\pi(z_1,z_2)=M_{z_2}T_{z_1}$ is  the \emph{time-frequency shift}, composition of the modulation $M_{z_2}$ and translation  $T_{z_1}$  defined by
$$M_{z_2}f(t)=e^{2\pi i z_2t}f(t),\quad T_{z_1}f(t)=f(z_1-t),\quad t,z_1,z_2\in\rd.$$

%The major advantage of the Wigner kernel versus the STFT one resides in  study of the %Schr\"{o}dinger equation with quadratic Hamiltonians.
The covariance property of $\cA$ is equivalent to being a member of the Cohen class for the related $\cA$-Wigner distribution (cf. \cite{Cohen1,Cohen2,Elena-book,grochenig}). In fact, we show (see Theorem \ref{Thaggiunta1}):
	\begin{equation*}
W_\cA (f,g)=W(f,g)\ast\sigma_\cA,\quad f,g\in\cS(\rd),
\end{equation*}
	where 
	\begin{equation}\label{5Z}
	\sigma_\cA=\cF^{-1}(e^{-\pi i \zeta\cdot B_\cA\zeta})\in \cS'(\rdd),
	\end{equation}
	and  $B_\cA$  is a symmetric $2d\times 2d$ matrix that can be computed explicitly from the covariant matrix $\cA$, cf.  \eqref{aggiunta3} in the sequel.
%We stress the reader that members of the Cohen class (see \cite{Cohen1,Cohen2} for a detailed account)  can be recaptured via $\cA$-Wigner representations.
%In this context, $\cA$-Wigner distributions can be viewed as \emph{perturbations} of the Wigner representation, so it is natural to study their difference in term of \emph{time-frequency content}. Under the invertibility assumption of  $B_\cA$ 
%(which allows to compute the inverse Fourier transform of $e^{-\pi i \zeta\cdot B_\cA\zeta}$ and  yields to a )
%we are able to show that the time-frequency content is the same, if measured in terms of modulation spaces (cf. Theorem \ref{charact mod}): for $1\leq p,q\leq\infty$,  $f\in\mathcal{S}'\left(\mathbb{R}^{d}\right)$, we have 
%	\[
%	Wf\in M^{p,q}\left(\rdd\right) \Leftrightarrow W_{\cA }f\in M^{p,q}\left(\rdd\right).
%	\]	
%Finally, $\cA$-Wigner representations can be efficiently  employed in the study of  Schr\"{o}dinger equations. 
The Cohen class will play a role for applications to {S}chr\"odinger equations; though, it presents two drawbacks  when looking at \eqref{5C}, \eqref{5D}. On one hand, it is too restrictive, since $\cA={\bf A_{ST}}$ in \eqref{A-FT} is not covariant, that is the short-time Fourier transform is excluded.  On the other hand, the matrix $\cA={\bf A}_\tau$ in \eqref{Aw-tau} is covariant for all $\tau\in\bR$, in particular for the forbidden Rihaczek cases $\tau=0,1$ for which \eqref{5C}, \eqref{5D} fail.  This suggests the introduction of the new class of {\bf shift-invertible} matrices $\cA\in Sp(2d,\bR)$ with related distributions $W_\cA$ satisfying (Definition \ref{semi-covariant})
\begin{equation}\label{5E}
|W_{\cA}(\pi(w)f,g)|=|T_{E_\cA(w)}W_{\cA}(f,g)|,\quad f,g\in\lrd,\quad w\in\rdd,
\end{equation}
for some $E_\cA \in GL(2d,\bR)$, with 
\begin{equation}\label{5F}
T_{E_\cA(w)}W_{\cA}(f,g)(z)=W_{\cA}(f,g)(z-E_\cA w),\quad w,z\in\rdd.
\end{equation}
We prove that the shift-invertible distribution $W_\cA$ satisfies \eqref{5D} and 
\begin{equation}\label{5G}
f\in M^p_{v_s}(\rd) \Leftrightarrow W_{\cA}f \in L^p_{v_s}(\rdd).
\end{equation}
This provides a general characterization of the modulation spaces $M^p_{v_s}$, see Theorem \ref{mainE} and Corollary \ref{mainEcor} for precise statements and bounds on the values of $p$. 
Note that the matrix $\cA={\bf A_{ST}}$ in \eqref{A-FT} is shift-invertible, recapturing in this way the standard definition of modulation spaces. As far as the $\tau$-Wigner matrix $\cA={\bf A}_\tau$ concerns, it is shift-invertible for $\tau\in\bR\setminus\{0,1\}$. This can be read as an explanation of the anomaly of the Rihaczek distributions. \par The block decomposition of the shift-invertible matrix  $\cA$ and the corresponding matrix $E_\cA$ in \eqref{5E}, \eqref{5F} can be explicitly computed, cf. \eqref{matrixE} below, and we may characterize the relevant subclasses of the distributions $W_\cA$ which are simultaneously covariant and shift-invertible (Remark \ref{rem2.20}).  

Finally, we address to the more precise equivalence \eqref{5C} concerning the case of different indices $p,q$. We first reconsider the $\tau$-Wigner case, $\tau\in\bR\setminus\{0,1\}$, and extend, with respect to \cite{CR2021}, the validity of \eqref{3bis} to $0<p,q<\infty$. This example suggests a deeper study of the matrices $\cA\in Sp(2d,\bR)$ such that 
\begin{equation}\label{5H}
\mu(\cA)=\cF_2 \mathfrak{T}_{L}
\end{equation}
where $\cF_2$ is the partial Fourier transform with respect to the second variable and $\mathfrak{T}_{L}$ is the $L^2$-normalized change of variables defined by a $d\times d$  invertible matrix $L$, cf. \cite{CT2020}.  We characterize the subclass of all the  $\cA\in Sp(2d,\bR)$ which are covariant and shift-invertible (see Proposition \ref{E7} and subsequent remark). Namely,   for covariant shift-invertible matrices $\cA$ of the form \eqref{5H} we prove
\begin{equation}\label{5J}
f\in M^{p,q}(\rd)\Leftrightarrow W_\cA(f,g)\in L^{p,q}(\rdd)
\end{equation}
with equivalence of norms valid also in the weighted cases for $0<p,q\leq\infty$ (Theorem \ref{caso-A-Wigner}). \par
A further analysis concerns the covariant case (Wigner perturbations, according to the terminology of \cite{CT2020}). If $\cA$ is covariant of the form \eqref{5H} then 
\begin{equation}\label{5K}
W_\cA (f,g)=W(f,g)\ast\sigma_\cA\quad f,g\in\cS(\rd),
\end{equation}
where $\sigma_\cA$ has now the particular form (see Corollary \ref{CorThaggiunta1}). We perform a detailed study of such convolution kernel (Lemma \ref{chirp spaces}, Proposition \ref{kernel spaces}). In particular, we deduce 
\[
Wf\in M^{p,q}\left(\rdd\right) \Leftrightarrow W_{\cA }f\in M^{p,q}\left(\rdd\right), \quad 1\leq p,q\leq\infty
\]
(see Theorem \ref{charact mod} for weighted versions of the above equivalence).

Besides providing a characterization for modulation spaces, the introduction of the $\cA$-Wigner distributions is strongly motivated by the applications to Schr\"{o}dinger equations. Let us first recall some classical results for the case of the quadratic Hamiltonians. 

Namely, consider 
\begin{equation}\label{C12}
\begin{cases} i \displaystyle\frac{\partial
	u}{\partial t} +\Opw(H) u=0\\
u(0,x)=u_0(x).
\end{cases}
\end{equation}
where $\Opw(H)$ is the Weyl quantization of a real quadratic polynomial in $\rdd$:
\begin{equation}\label{I18}
H\phas= \frac12 x Ax + \xi B x+ \frac12 \xi C\xi 
\end{equation}
with $A,C$ symmetric and $B$ invertible. We consider the Hamiltonian system 
\begin{equation}\label{C13}
\begin{cases}
2\pi \dot x=\nabla  _\xi H  =Bx+C\xi,\quad x(0)=y\\
2\pi \dot \xi=-\nabla _x H =-A x-B^T\xi,\quad \xi(0)=\eta,
\end{cases}
\end{equation}
with Hamiltonian matrix 
 $$\mathbb{D}:=\begin{pmatrix}B&C\\
-A&-B^T\end{pmatrix}\in \mathrm{sp} (d,\bR )$$
($\mathrm{sp} (d,\bR )$ is the symplectic algebra). We have, for $t\in\bR$, 
$\chi_t=e^{t\mathbb{D}}\in Sp(d,\R)$ 
and a solution to \eqref{C13} is given by $\phas= \chi_t (y,\eta)$.

The problem \eqref{C12} is solved by the Schr\"odinger propagator $$u(t,x)=e^{it\Opw(H)}u_0(x)=\mu(\chi_t)u_0$$ for a continuous choice of the phase factor in the definition of $\mu(\chi_t)$. If $u_0\in\lrd$ then $u(t,x)\in\lrd$, for every $t\in\bR$, see for example the textbooks \cite{folland,Birkbis}, whereas in the Lebesgue spaces $L^p(\rd)$, $p\not=2$, the soulution $u(t,x)$ does not keep the order of regularity of the initial datum $u_0$. \par 
Modulation spaces reveal here their effectiveness, in fact from Theorem \ref{Teorema0} (see also \cite{grochenig} and \cite{Elena-book}) we have that $u_0\in M^p_{v_s}(\rd)$ implies $u(t,\cdot)\in M^p_{v_s}(\rd)$, for every $0<p<\infty$, $s\geq0$.\par
Returning now to the subject of the present paper, let us recall from the original work of Wigner \cite{Wigner32} (see also \cite{BM49}):

\emph{The Wigner transform with respect to the space variable $x$ of the solution $u(t,x)$ of \eqref{C12} is given by }
\begin{equation}\label{C14}
Wu(t,z)=Wu_0(\chi^{-1}_tz),\quad z=\phas\in\rdd,\,t\in\bR.
\end{equation}	
It is natural  to replace the Wigner transform in \eqref{C14} with more general distributions by keeping  the action of the classical Hamiltonian flow $\chi_t$. A general result is easily obtained in the framework of the Cohen classes $Q_\sigma f= Wf\ast\sigma$, for any $\sigma\in\cS'(\rdd)$. Namely, assuming $u\in\cS(\rd)$, we have (Theorem \ref{Thrm-schrodinger})
\begin{equation}\label{C14bis}
Q_\sigma(u(t,\cdot))(z)=Q_{\sigma_t}(u_0)(\chi_t^{-1}z),\quad z=\phas\in\rdd,\,t\in\bR.
\end{equation}	
where $\sigma_t(z)=\sigma(\chi_t z)$. 
Note that in \eqref{C14bis} the Cohen class $Q_{\sigma_t}$ in the right-hand side depends on the time $t$. We may as well keep $Q_\sigma(u_0)$ for a fixed $\sigma$ in the right, and transfer the dependence on $t$ to the left. The classical Wigner case in \eqref{C14} corresponds to the choice $\sigma=\delta$ for which $\sigma_t(z)=\delta(\chi_t z)=\delta$, for every $z\in\rdd$.\par
Willing to give a precise functional setting to \eqref{C14bis} in the framework of modulation spaces, we limit attention to Cohen distributions generated by covariant matrices $\cA \in Sp(2d,\bR)$, $Q_\sigma u=W_\cA u=Wu\ast \sigma_\cA$, with kernel $\sigma_\cA$ given by \eqref{5Z}. The identity \eqref{C14bis} then reads (Proposition \ref{simpl-cov}):
\begin{equation}\label{C14A}
Q_\sigma(u(t,\cdot))(z)=W_{\cA}(u(t,\cdot))(z)= W_{\cA_t}(u_0)(\chi_t^{-1}z),
\end{equation}
where $\cA_t\in Sp(2d,\bR)$ is covariant for all $t\in\bR$, with Cohen kernel
$$\sigma_{\cA_t}(z)=\cF^{-1}\left(e^{-\pi i \zeta\cdot B_{\cA_t}\zeta }\right)(z),$$
$$ B_{\cA_t}= (\chi_t^{-1})^T B_\cA \chi_t^{-1},$$
$B_{\cA}$ as in \eqref{5Z}, cf. \eqref{aggiunta3}. Taking then $u_0\in M^p_{v_s}(\rd)$, $1\leq p\leq 2$, $s\geq 0$, we have from \eqref{5B}, cf. Corollary  \ref{Cor3.14}:
\begin{equation}\label{C14B}
W_\cA (u(t,\cdot)) \in M^p_{v_s}(\rdd),\quad W_{\cA_t}u_0\in M^p_{v_s}(\rdd),\quad t\in\bR,
\end{equation}
and each one of these conditions is equivalent to the assumption $u_0 \in M^p_{v_s}(\rd)$. Willing to have instead
\begin{equation}\label{C14C}
W_\cA (u(t,\cdot)) \in L^p_{v_s}(\rdd),\quad W_{\cA_t}u_0\in L^p_{v_s}(\rdd),\quad t\in\bR,
\end{equation}
 we are led to assume that the matrix $\cA$ is also shift-invertible. In Proposition \ref{Prop4.5} we shall prove that $\cA$ is shift-invertible if and only if $\cA_t$  is shift-invertible, for any fixed $t\not=0$. Hence in this case the conditions \eqref{C14C} are equivalent to $u_0\in M^p_{v_s}(\rd)$. As an example, we shall test these results on the free particle. 
The property of regularity \eqref{C14C} is the starting point for a proceeding in localization similar to that in \cite{CR2021}. Namely, cf. Definition \ref{4.6}, for a covariant and shift-invertible $\cA$ we define for $f\in\lrd$ the generalized Wigner wave front set $\mathcal{W}\cF^{p,s}_{\cA}(f)$, $1\leq p\leq 2$, $s\geq 0$, by setting $z_0=(x_0,\xi_0)\notin \mathcal{W}\cF^{p,s}_{\cA}(f)$, $z_0\not=0$, if there exists a conic neighbourhood $\Gamma_{z_0}\subset\rdd$  such that 
\begin{equation}\label{26A}
\int_{\Gamma_{z_0}} \la z\ra^{ps} |W_{\cA}f(z)|^p\,dz<\infty.
\end{equation}
We have from \eqref{5G} that $\mathcal{W}\cF^{p,s}_{\cA}(f)=\emptyset$ if and only if $f\in M^p_{v_s}(\rd)$, cf. Proposition \ref{4.7}. For the standard Wigner transform  the notation  $\mathcal{W}\cF^{p,s}_{\cA_{1/2}}(f)$, cf. \eqref{Aw-tau}, will be shortened to $\mathcal{W}\cF^{p,s}(f)$. From \eqref{C14A} and \eqref{C14C} we deduce the following propagation of micro-singularities for the solutions of \eqref{C12}, cf. Theorem \ref{4.8}:
\begin{equation}\label{26B}
\mathcal{W}\cF^{p,s}_{\cA}(u(t,\cdot))=\chi_t(\mathcal{W}\cF^{p,s}_{\cA_t}(u_0)), 
\end{equation}
in particular for the standard Wigner transform
\begin{equation}\label{26C}
\mathcal{W}\cF^{p,s}(u(t,\cdot))=\chi_t(\mathcal{W}\cF^{p,s}(u_0)).  
\end{equation}
We address to the forthcoming second part of \cite{CR2021} for a detailed study of $\mathcal{W}\cF^{p,s}_{\cA}$ with applications to Fourier integral operators and Schr\"odinger equations of more general type. We limit here to the following warning and remarks. First, we cannot extend to the Wigner wave front set all the properties of the classical wave front set of H\"ormander, cf. \cite{HormanderI} or its global version \cite{hormanderglobalwfs91}. In fact, the inclusion of the wave front set of the solutions in the characteristic manifold, for a homogeneous  linear partial differential equation, is false for the Wigner wave front. This depends on the existence of the ghost frequencies, see the final comments in \cite{CR2021}. On the other hand, the whole Wigner wave front, including its ghost part, is exactly preserved by the Schr\"odinger propagator, as clarified by  \eqref{26B} and \eqref{26C}.
\section {Time-frequency analysis tools}
\textbf{Notations.} We set $t^2=t\cdot t$,  $t\in\rd$, and
$xy=x\cdot y$ is the scalar product on $\Ren$.  The space   $\sch(\Ren)$ denotes the Schwartz class whereas $\sch'(\Ren)$  the space of temperate distributions.   The brackets  $\la f,g\ra$ denote the extension to $\sch' (\Ren)\times\sch (\Ren)$ of the inner product $\la f,g\ra=\int f(t){\overline {g(t)}}dt$ on $L^2(\Ren)$ (conjugate-linear in the second component). The \emph{reflection operator}  $\cI$ is given by $\cI f(t)= f(-t).$
The Fourier transform is normalized to be
 $${\hat
	{f}}(\o)=\Fur f(\o)=\intrd
f(t)e^{-2\pi i t\o}dt.$$
The  symplectic matrix 
\begin{equation}\label{J}
J=\begin{pmatrix} 0_{d\times d}&I_{d\times d}\\-I_{d\times d}&0_{d\times d}\end{pmatrix},
\end{equation}
(here $I_d$, $0_d$ are  the $d\times d$ identity matrix and null matrix, respectively)
enters the definition of   the standard symplectic form $\sigma(z,z')=Jz\cdot z'$. They allow to introduce  the symplectic Fourier transform:
\begin{equation}\label{C1SFT}
\cF_\sigma a(z)=\intrdd e^{-2\pi i \sigma(z,z')} a(z')\,dz'.
\end{equation}

The Fourier transform and symplectic Fourier transform are related by
\begin{equation}\label{C1SFT-FT}
\cF_\sigma a(z)=\cF a(Jz)=\cF(a\circ J)(z), \quad a\in\cS(\rdd).
\end{equation}
For the study of perturbations of the Wigner distribution we will use the Ambiguity Function 
$Amb\left(f\right)$ defined as 
\begin{equation}\label{ambiguity}
Amb\left(f\right)\phas=\cF_\sigma (W f)\phas=\intrd f\left(y+\frac x2\right) \overline{f\left(y-\frac x2\right)}e^{-2\pi i y\xi} dy.
\end{equation}

We denote by $GL(2d,\bR)$  the linear group of $2d\times 2d$ invertible matrices; for a complex-valued function $F$ on $\rdd$ and $L\in GL(2d,\bR)$ we define
\begin{equation}\label{Ltransf}
\mathfrak{T}_{L} F(x,y)=\sqrt{|\det L|}F(L(x,y)), \quad (x,y)\in\rdd,
\end{equation}
with the convention 
$$L(x,y)= L \left(\begin{array}{c}
x\\
y
\end{array}\right),\quad (x,y)\in\rdd.$$

For $1\leq p\leq\infty$, the spaces $\ell^\infty_{mn}\ell^p_{m'n'}$
are the Banach
spaces of sequences $\{a_{m',n',m,n}\}$ such that
 such that
$$\|a_{m',n',m,n}\|_{\ell^\infty_{mn}\ell^p_{m'n'}}:=\sup_{m,n\in\zd}\left(\sum_{m',n'\in\zd}
|a_{m',n',m,n}|^p\right)^{1/p}<\infty
$$
(with obvious changes when $p=\infty$).
\subsection{Modulation
spaces}

In this paper  $v$ is  a continuous, positive,  submultiplicative  weight function on $\rd$, i.e., 
$ v(z_1+z_2)\leq v(z_1)v(z_2)$, for all $ z_1,z_2\in\Ren$.
A weight function  $m$ is in $\mathcal{M}_v(\rd)$ if $m$ is a positive, continuous  weight function  on $\Ren$ and it is {\it
	$v$-moderate}:
$ m(z_1+z_2)\leq Cv(z_1)m(z_2)$  for all $z_1,z_2\in\Ren$.

 In the following we will work with  weights on $\rdd$ of the type
\begin{equation}\label{weightvs}
v_s(z)=\la z\ra^s=(1+|z|^2)^{s/2},\quad z\in\rdd,
\end{equation}

%When it is necessary to be clear about the weights' dimension, we shall write
%$$ v_{s,d}:=v_s \,\,\mbox{in\,dimension}\,\,2d,\quad v_{s}:= v_s  \,\,\mbox{in\,dimension}\,\,4d.$$
For $s<0$, $v_s$ is $v_{|s|}$-moderate.\par 
For weight functions $m_1,m_2$ on $\rd$, we will  use the notation $$(m_1\otimes m_2)(x,\o)=m_1(x)m_2(\o),\quad x,\o\in \rd,$$
and similarly for weights $m_1,m_2$ on $\rdd$. In particular, we shall use the weight functions on $\bR^{4d}$:
\begin{equation}\label{weight1tensorvs}
( v_s\otimes 1)(z,\zeta)=(1+|z|^2)^{s/2},\quad (1\otimes v_s)(z,\zeta)=(1+|\zeta|^2)^{s/2},\quad z,\zeta\in\rdd.
\end{equation}
The modulation spaces,  introduced by Feichtinger in \cite{feichtinger-modulation} and extended to the  quasi-Banach setting Galperin and Samarah   \cite{Galperin2004}, are  now available in many textbooks, see e.g. \cite{KB2020,Elena-book,grochenig}.

	Fix a non-zero window $g$ in the Schwartz class $\cS(\rd)$. Consider  a weight function $m\in\mathcal{M}_v$ and indices $0<p,q\leq \infty$. The modulation space $M^{p,q}_m(\rd)$ is the subspace of tempered  distributions $f\in\cS'(\rd)$ with
\begin{equation}\label{norm-mod}
\|f\|_{M^{p,q}_m}=\|V_gf\|_{L^{p,q}_m}=\left(\intrd\left(\intrd |V_g f \phas|^p m\phas^p dx  \right)^{\frac qp}d\o\right)^\frac1q <\infty
\end{equation}
(natural changes  with $p=\infty$ or $q=\infty)$. 
We write $M^p_m(\rd)$ for $M^{p,p}_m(\rd)$ and $M^{p,q}(\rd)$ if $m\equiv 1$.

For $1\leq p,q\leq \infty$, the space $M^{p,q}(\mathbb{R}^d)$ is a Banach space whose definition is
independent of the choice of the window $g$:  \emph{different
non-zero window functions in $\cS(\rd)$ yield equivalent norms}.  The window class can be extended to the modulation space $M^1_v(\rd)$ (Feichtinger algebra). The  modulation space $M^{\infty,1}(\rd)$ coincides with the Sj\"ostrand's class in \cite{Sjostrand1}.

We recall their  inclusion properties: 
\begin{equation}
\mathcal{S}(\mathbb{R}^{d})\subseteq M^{p_{1},q_{1}}_m(\mathbb{R}%
^{d})\subseteq M^{p_{2},q_{2}}_m(\mathbb{R}^{d})\subseteq \mathcal{S}^{\prime
}(\mathbb{R}^{d}),\quad p_{1}\leq p_{2},\,\,q_{1}\leq q_{2}.
\label{modspaceincl1}
\end{equation}%

Denoting by $%
\mathcal{M}_m^{p,q}(\mathbb{R}^d)$ the closure of $\mathcal{S}(\mathbb{R}^d)$ in the $M^{p,q}_m$-norm, we observe
\begin{equation*}
\mathcal{M}_m^{p,q}(\mathbb{R}^d) \subseteq M^{p,q}_m(\mathbb{R}^d), \quad 0<p,q\leq\infty,
\end{equation*}
 and 
 \begin{equation*} \mathcal{M}^{p,q}_m (\mathbb{R}^d) = M^{p,q}_m (\mathbb{R}^d),\quad 0<p,q<\infty.
\end{equation*}

For $m,w\in\mathcal{M}_v(\rd)$, the Wiener amalgam spaces $W(\mathcal{F}L^{p}_m,L^{q}_w)(\mathbb{R}^{d})$ can be viewed as images under \ft\, of the modulation spaces. Namely, for $p,q\in (0,\infty]$,  $f\in \mathcal{S}^{\prime }(\mathbb{R}^{d})$ belongs to $W(\mathcal{F}L^{p}_m,L^{q}_w)(\mathbb{R}^{d})$ if
\begin{equation*}
\Vert f\Vert _{W(\mathcal{F}L^{p}_m,L^{q}_w)(\mathbb{R}^{d})}:=\left( \int_{%
	\mathbb{R}^{d}}\left( \int_{\mathbb{R}^{d}}|V_{g}f(x,\o )|^{p}\,m(\o)^pd\o
\right) ^{q/p}w(x)^qdx\right) ^{1/q}<\infty \,
\end{equation*}%
(obvious modifications for $p=\infty $ or $q=\infty $). Using the \emph{fundamental identity of time-frequency analysis} \cite[formula (1.31)]{Elena-book} \begin{equation}\label{FI-TFA}
V_{g}f(x,\o )=e^{-2\pi ix\o
}V_{\hat{g}}\hat{f}(\o ,-x),\end{equation}
we can deduce
\begin{equation*}
|V_{g}f(x,\o )|=|V_{\hat{g}}\hat{f}(\o ,-x)|=|\mathcal{F}(\hat{f}%
\,T_{\o }\overline{\hat{g}})(-x)|
\end{equation*}%
so that 
\begin{equation*}
\Vert f\Vert _{{M}_{m\otimes w}^{p,q}}=\left( \int_{\mathbb{R}^{d}}\Vert \hat{f}\
T_{\o }\overline{\hat{g}}\Vert _{\mathcal{F}L^{p}_{v}}^{q}m(\o)\ d\o
\right) ^{1/q}=\Vert \hat{f}\Vert _{W(\mathcal{F}L^{p}_{m},L^{q}_{w})}.
\end{equation*}
The above equality of norms yields
\begin{equation}
\mathcal{F}({M}^{p,q}_{v\otimes w})=W(\mathcal{F}L^{p}_{v},L^{q}_{w}).  \label{W-M}
\end{equation}

%We shall use 	the following properties:
%\begin{itemize}
%	\item [(i)] Inclusion relations: if $0<p_1\leq p_2\leq \infty$,  $0<q_2\leq q_1\leq \infty$, then
%	\begin{equation}\label{inclusionW}
%	W(L^{p_2}, L^{q_2}_m)(\rd)\hookrightarrow W(L^{p_1}, L^{q_1}_m)(\rd).
%	\end{equation}
%	\item [(ii)] Convolution relations (cf. \cite[Lemma 2.9]{Galperin2004} for the quasi-Banach case): Consider $m_i\in\mathcal{M}_v$, $0<p_i,q_i\leq \infty$, $i\in \{1,2,3\}$, and $p_3\geq 1$.  Assume that $L^{p_1} \ast L^{p_2} \hookrightarrow L^{p_3}$ and $\ell^{q_1}_{m_1}\ast \ell^{q_2}_{m_2} \hookrightarrow \ell^{q_3}_{m_3}$, then
%	\begin{equation}\label{convWiener}
%	W(L^{p_1}, L^{q_1}_{m_1})\ast W(L^{p_2}, L^{q_2}_{m_2})\hookrightarrow W(L^{p_3}, L^{q_3}_{m_1}).
%	\end{equation}
%	\item[(iii)] If $m\in\mathcal{M}_v$, $0<p\leq\infty$, then \begin{equation}\label{pWiener}
%	L^p_m=W(L^p,L^p_m).
%	\end{equation}
%\end{itemize}

\subsection{The metaplectic
representation} Recall the symplectic group 
\begin{equation}\label{defsymplectic}
\Spnr=\left\{\cA\in\Gltwonr:\;\cA^T J\cA=J\right\},
\end{equation}
where $\cA^T$ denotes the transpose of $\cA$ and the symplectic matrix
$J$ is defined in \eqref{J}. In the sequel, we shall also refer to symplectic matrices in \emph{double dimension},
induced from the standard symplectic form on $\bR^{4d}$: 
\begin{equation}\label{defsymplectic4d}
	Sp(2d,\bR)=\left\{\cA\in GL(4d,\bR):\;\cA^T J\cA=J\right\},
\end{equation}
where $J$ is the one in \eqref{J} with $0_{d\times d}$ replaced by  $0_{2d\times 2d}$ and $I_{d\times d}$ replaced by $I_{2d\times 2d}$.

 The metaplectic representation
$\mu$ is a unitary
representation of the (double
cover of the) symplectic
group $\Spnr$ on $\lrd$.
The symplectic  algebra
$\spdr$ is the set of all
$2d\times 2d$ real matrices
$\A$   such that $e^{t \A}
\in \Spnr$ for all
$t\in\R$.\par
 For
some elements of $\Spnr$  the metaplectic
representation can be
computed explicitly.  Namely,  using  the notations in \cite{Birkbis,deGossonQHA21}, for
$f\in L^2(\R^d)$,  $C$ real symmetric $d\times d$ matrix ($C^T=C$) we have, up to a phase factor $s$ (that is, $|s|=1$),
\begin{equation}
\mu(J)f=\cF f\label{muJ};
%\mu\left(\begin{pmatrix} A&0_{d\times d}\\
%0_{d\times d}&\;^t\!A^{-1}\end{pmatrix}\right)f(x)
%&=(\det A)^{-1/2}f(A^{-1}x)
\end{equation}
for $$V_C:=\begin{pmatrix} I_{d\times d}&0\\
C&I_{d\times d}\end{pmatrix},$$
up to a phase factor
\begin{equation}
\mu\left(V_C\right)f(x)= e^{i\pi 
Cx\cdot x}f(x)\label{lower}.
\end{equation}

Special instances of metaplectic operators also called \emph{rescaling operators}. They are metaplectic operators $\mu(\cD_L)$ associated with the symplectic matrix $\cD_L$ constructed as follows.  For any $L\in GL(d,\bR)$,
\begin{equation}\label{MotL}
\cD_L= \left(\begin{array}{cc}
L^{-1} &0_{d\times d}\\
0_{d\times d} & L^T
\end{array}\right)  \in Sp(d,\bR).
\end{equation}
Then,  up to a phase factor, 
\begin{equation}\label{AdL}
\mu(\cD_L)F(x)=\sqrt{|\det L|}F(Lx)=\mathfrak{T}_{L} F(x),\quad F\in\lrd.
\end{equation}
The metaplectic operators posses a group structure called the metaplectic group.
\begin{proposition}\label{deGosson96}
	The metaplectic group is generated by the operators $\mu(J), \mu(\cD_L)$ and $\mu(V_C)$.
\end{proposition}
In the paper we shall work   both with the symplectic group  $Sp(d,\bR)$ of  $2d\times 2d$  matrices and  $Sp(2d,\bR)$ of  $4d\times 4d$ ones. In particular, the  matrix $\cA\in Sp(2d,\bR)$ is assumed to have the $4\times 4$ block decomposition of $2d\times 2d$ matrices: 
\begin{equation}\label{blockmatrixA}\mathcal{A}=\begin{pmatrix}
A&B\\C&D\end{pmatrix}\end{equation}
with the  decompositions of the $2d\times 2d$ sub-blocks as follows:
\begin{equation}\label{sub-block}A=\begin{pmatrix}
A_{11}&A_{12}\\A_{21}&A_{22}\end{pmatrix},\quad B=\begin{pmatrix}
B_{11}&B_{12}\\B_{21}&B_{22}\end{pmatrix},\quad C=\begin{pmatrix}
C_{11}&C_{12}\\C_{21}&C_{22}\end{pmatrix}, \quad D=\begin{pmatrix}
D_{11}&D_{12}\\D_{21}&D_{22}\end{pmatrix}.
\end{equation}

\begin{definition}\label{def4.1}
	For a $4d\times 4d$ symplectic matrix $\cA\in Sp(2d,\bR)$ we
	 define the time-frequency representation \textbf{$\cA$-Wigner}  by 
	\begin{equation}\label{WignerA}
	W_\cA (f,g)=\mu(\cA) (f\otimes \bar{g}),\quad f,g\in\lrd.
	\end{equation}
	We set $W_\cA f:= W_\cA (f,f)$.
\end{definition}
% The following
%formulae for the metaplectic
%representation can be found
%in  \cite[Theorems 4.51 and
%4.53]{folland}.
\subsubsection{Properties of ${W}_{\cA}\left(f,g\right)$.} In what follows we list all the elementary properties enjoyed by the $\cA$-Wigner distribution. The continuity of ${W}_{\cA}$  was shown in \cite{CR2021}:
\begin{proposition}
	\label{def bilA triple} Assume $\cA\in Sp(2d,\bR)$. Then,
	\begin{enumerate}
		\item If $f,g\in L^{2}(\mathbb{R}^{d})$, then ${W}_{\cA}(f,g)\in L^{2}(\rdd)$
		and the mapping ${W}_{\cA}:L^{2}(\mathbb{R}^{d})\times L^{2}(\mathbb{R}^{d})\rightarrow L^{2}(\rdd)$
		is continuous. 
		\item If $f,g\in\mathcal{S}(\mathbb{R}^{2d})$, then ${W}_{\cA}(f,g)\in\mathcal{S}(\mathbb{R}^{2d})$
		and the mapping  ${W}_{\cA}:\mathcal{S}(\mathbb{R}^{d})\times\mathcal{S}(\mathbb{R}^{d})\rightarrow\mathcal{S}(\rdd)$
		is continuous. 
		\item If $f,g\in\mathcal{S}'(\mathbb{R}^{d})$, then ${W}_{\cA}(f,g)\in\mathcal{S}'(\rdd)$
		and the mapping  ${W}_{\cA}:\mathcal{S}'(\mathbb{R}^{d})\times\mathcal{S}'(\mathbb{R}^{d})\rightarrow\mathcal{S}'(\rdd)$
		is continuous. 
	\end{enumerate}
\end{proposition}

\begin{proposition}[Interchanging $f$ and $g$]
	\label{interchanging f g}
 For $\cA \in Sp(2d,\bR)$ with block decomposition \eqref{blockmatrixA} and   $f,g\in L^{2}\left(\mathbb{R}^{d}\right)$.
	Then
	\[
	{W}_{\cA}(g,f)={W}_{\widetilde{\cA}}(\bar{f},\bar{g}),
	\]
	where 
	$\mathcal{\widetilde{\cA}}=\begin{pmatrix}
	A L&BL\\CL&DL \end{pmatrix}$ and 
	 \begin{equation}\label{L}
	L=\left(\begin{array}{cc}
	0_{d\times d} & I_{d\times d}\\
	I_{d\times d} &0_{d\times d}
	\end{array}\right).
	\end{equation}
	Precisely, using the sub-block decomposition \eqref{sub-block}, we obtain
	$AL=\begin{pmatrix}
	A_{12}&A_{11}\\A_{22}&A_{21}\end{pmatrix}$ and similarly for the other block matrices $B,C,D$.
\end{proposition}
\begin{proof}
Consider the matrix $L$ defined in \eqref{L} and observe that $L^T=L^{-1}=L$. The symplectic matrix 	$\cD_L$ in \eqref{MotL} becomes
\begin{equation*}
	\cD_L= \left(\begin{array}{cc}
	L^{-1} &0_{d\times d}\\
	0_{d\times d} & L^T
	\end{array}\right)=  \left(\begin{array}{cc}
	L &0_{d\times d}\\
	0_{d\times d} & L
	\end{array}\right). 
	\end{equation*}
With our choice of $L$, 
$$\mu(\cD_L) (g\otimes\bar{f})(x,y)=(g\otimes\bar{f})(y,x)=\bar{f}\otimes g(x,y).$$
This let us factorize $W_{\cA}(g,f)$ as follows:
$$W_{\cA}(g,f)(x,y)=\mu(\cA)(g\otimes\bar{f})(x,y)=\mu(\cA\cD_L^{-1}\cD_L)(g\otimes\bar{f})(x,y)=\mu(\cA\cD_L)(\bar{f}\otimes g)(x,y),$$
and the claim easily follows by observing that $\cA\cD_L=\widetilde{A} $.
\end{proof}

We now limit ourselves  to  matrices $\cA \in Sp(2d,\bR)$ such that 
\begin{equation}\label{metapf2dil}
\mu(\cA)=\cF_2 \mathfrak{T}_{L}
\end{equation}
where $\cF_2$ is the partial Fourier transform with respect to the second variables $y$ defined by
\begin{equation}\label{FT2}
\cF_2 F(x,\o)=\intrd e^{-2\pi i y\cdot \o}F(x,y)\,dy,\quad F\in L^2(\rdd).
\end{equation} 
and the change of coordinates $\mathfrak{T}_{L}$ is defined in \eqref{Ltransf}.  The following fact was established in \cite[Proposition 3.3]{CT2020}:
\begin{proposition}\label{change f g 2}	
	For $f,g\in L^{2}\left(\mathbb{R}^{d}\right)$, $\mu(\cA)$ of the form \eqref{metapf2dil} with
$$	L=\left(\begin{array}{cc}
L_{11} &L_{12}\\
L_{21} & L_{22}
\end{array}\right),$$ 
	then
	\[
	W_\cA\left(g,f\right)\left(x,\omega\right)=\overline{W_{\mathcal B}\left(f,g\right)\left(x,\omega\right)},
	\]
	with $\mu({\mathcal B})=\cF_2 \mathfrak{T}_{\tilde{L}}$, with
$$	\tilde{L}=\left(\begin{array}{cc}
L_{21} &-L_{22}\\
L_{11} & -L_{12}
\end{array}\right).$$	
\end{proposition}
More generally, 
\begin{proposition}
For $\cA \in Sp(2d,\bR)$, we have
$$W_\cA(g,f)=\overline{W_{\cB}(f,g)},$$
for a suitable $B\in  Sp(2d,\bR)$.
\end{proposition}
\begin{proof}
We use Proposition \ref{deGosson96}, and observe that  $\overline{\mu(J)f}=\mu{(J^{-1})}\bar{f}$, $\overline{\mu(V_C)f}=\mu{(V_{-C})}\bar{ f}$ and $\overline{\mu(\cD_L)f}=\mu(\cD_L)\bar{ f}$. This gives the claim. 
\end{proof}

What follows can be viewed as a generalization of the 	\emph{fundamental identity of time-frequency analysis for the STFT}, cf. \cite[(1.31)]{Elena-book}.
\begin{proposition}[Fundamental identity of time-frequency analysis]\label{Prop2.7}
	
 For $\cA \in Sp(2d,\bR)$ with block decomposition \eqref{blockmatrixA} and   $f,g\in L^{2}\left(\mathbb{R}^{d}\right)$,
	then
	\[
	W_{\cA}\left(\hat{f},\hat{g}\right)=W_{\widetilde{\cA}}\left({f},{g}\right),
	\]
	where 
	$\mathcal{\widetilde{\cA}}=\begin{pmatrix}
	B L&AL\\DL&CL \end{pmatrix}$ and 
	\begin{equation}\label{L1}
	L=\left(\begin{array}{cc}
	I & 0\\
	0 & -I
	\end{array}\right) 
	\end{equation}
\end{proposition}
\begin{proof}
	Using the reflection operators $\cI g(t)=g(-t)$, we can write $$\hat{f}\otimes\overline{\hat{g}}=\hat{f}\otimes\widehat{\cI\hat{\bar{g}}}=\mathcal{F}
	\mathfrak{T}_{L}\left(f\otimes \bar{g}\right)$$
	where 	 $L$ is defined in \eqref{L1}.
  Hence, 
	\[
	W_{\cA}(\hf,\hat{g})=\mu(\cA)(\hat{f}\otimes\overline{\hat{g}})=\mu(\cA)\mathcal{F}
	\mathfrak{T}_{L}\left(f\otimes \bar{g}\right)=\mu(\cA J \cD_L)(f\otimes \bar{g}).
	\]
The conclusion is a simple computation.
\end{proof}
\begin{proposition}[Fourier transform of $W_{\cA}$]
	\label{ft of btfd}
	Let $\cA\in  Sp(2d,\bR)$ and $f,g\in L^{2}\left(\mathbb{R}^{d}\right)$.
	Then, 
	\begin{equation}\label{FouB}
	\mathcal{F}{W}_{\cA}\left(f,g\right)={W}_{\widetilde{A}}\left(f,g\right),
	\end{equation}
	where 
$\widetilde{A}=(\cA^T)^{-1}J$.
\end{proposition}
\begin{proof}
	Since, up to a phase factor,  $	\mathcal{F}\mu(\cA)=\mu(J\cA)$, the result follows from the symplectic group property \eqref{defsymplectic}.
\end{proof}

\begin{proposition}[Moyal's Identity]
	\label{Moyal}
	Let $\cA\in  Sp(2d,\bR)$ and $f_1,f_2,g_1,g_2\in L^{2}\left(\mathbb{R}^{d}\right)$.
	Then, 
	\begin{equation}\label{MoyalF}
	\la  {W}_{\cA}\left(f_1,g_1\right), {W}_{A}\left(f_2,g_2\right)\ra_{L^2{(\bR^{2d})}}=\la f_1,f_2\ra_{L^2(\rd)} \overline{\la g_1,g_2\ra}_{L^2(\rd)},
	\end{equation}
in particular, for $f_1=f_2=f$, $g_1=g_2=g$,
$$\la  {W}_{\cA}\left(f,g\right), {W}_{A}\left(f,g\right)\ra_{L^2{(\bR^{2d})}}=\|f\|^2_{L^2(\rd)}\|g\|^2_{L^2(\rd)}.$$
\end{proposition}
\begin{proof}
	We simply use that $\mu(\cA)$ is unitary on $L^2(\rdd)$:
\begin{align*} \la  {W}_{\cA}\left(f_1,g_1\right), {W}_{A}\left(f_2,g_2\right)\ra_{L^2(\bR^{2d})}&=\la  \mu(\cA)\left(f_1\otimes \bar{g_1}\right), \mu(\cA)\left(f_2\otimes \bar{g_2}\right)\ra_{L^2(\bR^{2d})}\\
&=	\la  \mu(\cA)^{-1} \mu(\cA)\left(f_1\otimes \bar{g_1}\right),\left(f_2\otimes \bar{g_2}\right)\ra_{L^2(\bR^{2d})},
	\end{align*}
	and the claim follows.
\end{proof}

A simple computation shows the following  \emph{polarization identity}: 
\begin{equation}\label{pol}
W_\cA(f+g)=W_\cA(f)+W_\cA(g)+W_\cA(f,g)+W_\cA(g,f).
\end{equation}

The Covariance Property of \cite[Proposition 4.3]{CR2021} can be generalized  and improved as follows:
\begin{proposition}[Covariance Property]\label{C4covprop}
	Consider $\cA\in Sp(2d,\bR)$ having block decomposition
	\begin{equation*} \cA=\left(\begin{array}{cccc}
	A_{11} & A_{12}&A_{13}&A_{14}\\
	A_{21}&A_{22}& A_{23} &A_{24}\\
	A_{31}&A_{32}&A_{33} &A_{34}\\
	A_{41}&A_{42}& A_{43} &A_{44}\\
	\end{array}\right)
	\end{equation*}
	with $A_{ij}$, $i,j=1,\dots,4$,  $d\times d$ real matrices. Then the representation $W_\cA$ in \eqref{WignerA} is covariant, namely
	\begin{equation}\label{C4covariance}
	W_\cA(\pi(z)f,\pi(z)g)= T_z W_\cA (f,g), \quad  f,g\in\cS(\rd),\quad  z\in\rdd,
	\end{equation}
	if and only if $\cA$ is of the form
\begin{equation}\label{A-covariant}
\cA=\left(\begin{array}{cccc}
A_{11} & I_{d\times d}-A_{11}&A_{13}&A_{13}\\
A_{21}&-A_{21}&I_{d\times d}- A^T_{11} &-A^T_{11}\\
0_{d\times d}&0_{d\times d}&I_{d\times d} &I_{d\times d}\\
-I_{d\times d}&I_{d\times d}& 0_{d\times d} &0_{d\times d}\\
\end{array}\right).
\end{equation}
	with $A_{13}=A_{13}^T$,  $A_{21}^T=A_{21}$.
	The result does not depend on the choice of the phase factor in the definition of $\mu(\cA)$ and $W_\cA$ in \eqref{WignerA}.
\end{proposition}
\begin{proof}
The equivalence of \eqref{C4covariance} and the matrix 
\begin{equation}\label{A-covariantvecchia}
\cA=\left(\begin{array}{cccc}
A_{11} & I_{d\times d}-A_{11}&A_{13}&A_{13}\\
A_{21}&-A_{21}&I_{d\times d}- A^T_{11} &-A^T_{11}\\
A_{31}&-A_{31}&A_{33} &A_{33}\\
A_{41}&-A_{41}& A_{43} &A_{43}\\
\end{array}\right).
\end{equation}
 is a straightforward generalization of the proof of \cite[Proposition 4.3]{CR2021}. We notice that in the last element of the second row of  \cite[Formula (108)]{CR2021}
the entry $A^T_{11}$ should be replaced by $-A^T_{11}$ as in \eqref{A-covariantvecchia}. We then use the matrix-block properties for symplectic matrices (see, e.g. \cite[Proposition 4.1]{folland}) to obtain \eqref{A-covariant}. First, the condition 
$$AB^T= BA^T$$
(where $A$ and $B$ are the $2d\times 2d$ blocks in \eqref{sub-block}) gives
$A_{13}^T=A_{13}$.
  The block property:
$$A^TC= C^TA$$ yields to  $A_{31}=	0_{d\times d}$ and $A_{21}^TA_{41}=A_{41}^TA_{21}$.
From $$B^TD=D^TB$$ we infer 
$A_{43}= 0_{d\times d}$. Condition $$A^TD-C^TB=I_{d\times d}$$
yields to
$A_{33}= I_{d\times d}$ and $A_{41}=-I_{d\times d}$, which, together with $A_{21}^TA_{41}=A_{41}^TA_{21}$, gives the symmetric property $A_{21}^T=A_{21}$.
\end{proof}

Similarly,  a matrix $\cA\in Sp(2d,\bR)$ having the block-decomposition in \eqref{A-covariant} is called \emph{covariant}.\par
If we introduce the real symmetric $2d\times 2d$ matrix 
\begin{equation}\label{aggiunta3}
B_\cA=\left(\begin{array}{cc}
A_{13} & \frac12 I_{d\times d} -A_{11}\\
\frac12 I_{d\times d} -A_{11}^T&-A_{21}
\end{array}\right),
\end{equation}
 the covariance property of $\cA$ can be viewed as Cohen class property as shown below. The proof is a straightforward generalization of \cite[Theorem 4.6]{CR2021}:
\begin{theorem}\label{Thaggiunta1}
	Let $\cA\in Sp(2d,\bR)$ be of the form \eqref{A-covariant}. Then
	\begin{equation}\label{aggiunta6}
	W_\cA (f,g)=W(f,g)\ast\sigma_\cA,\quad f,g\in\cS(\rd),
	\end{equation}
	where 
	\begin{equation}\label{aggiunta7}
	\sigma_\cA=\cF^{-1}(e^{-\pi i \zeta\cdot B_\cA\zeta})\in \cS'(\rdd),
	\end{equation}
	and  $B_\cA$  defined in \eqref{aggiunta3}.
\end{theorem}

\begin{proposition} For $z=(z_1,z_2)$, $u=(u_1,u_2)$, we have
	\begin{equation*}
	W_\cA(\pi(z)f,\pi(u)g)=M_{(\zeta_3,\zeta_4)}T_{(\zeta_1,\zeta_2)} W_\cA( f,g) \quad  f,g\in\cS(\rd),\quad  \zeta_i\in\rdd, \,i=1,\dots,4,
	\end{equation*}
	where
	\begin{align}\label{tras-mod}
	(\zeta_1,\zeta_2)&=(A_{11}z_1+(I-A_{11})u_1+A_{13}(z_2-u_2),A_{21}(z_1-u_1)+(I-A^T_{11})z_2-A^T_{11}u_2)\\
	(\zeta_3,\zeta_4)&=(A_{31}(z_1-u_1)+A_{33}(z_2-u_2),A_{41}(z_1-u_1)+A_{43}(z_2-u_2)).\notag
	\end{align}
\end{proposition}
\begin{proof}
Using the intertwining property (see e.g. Formula $(1.10)$ in \cite{Elena-book}) 
$$\pi(\cA \zeta)=c_\cA \mu(\cA)\pi(\zeta)\mu(\cA)^{-1},\quad \zeta\in\bR^{4d}$$
(where $c_\cA$ is a phase factor: $|c_\cA|=1$), we calculate
\begin{align*}
W_\cA(\pi(z_1,z_2)f,\pi(u_1,u_2)g)&=\mu(\cA)[\pi(z_1,u_1,z_2,-u_2)(f\otimes \bar{g})]\\
&=c_\cA^{-1}\pi(\cA(z_1,u_1,z_2,-u_2))W_\cA (f,g).
\end{align*}
The covariance of $W_\cA$  gives the matrix block-decomposition in \eqref{A-covariant} so that 
\begin{equation*}
\pi(\cA(z_1,u_1,z_2,-u_2)) =c_\cA T_{(\zeta_1,\zeta_2)} M_{(\zeta_3,\zeta_4)},
\end{equation*}
with $(\zeta_1,\zeta_2)\in\rdd$ and $(\zeta_3,\zeta_4)\in \rdd$ in \eqref{tras-mod}.
\end{proof}

Metaplectic operators are bounded on modulation spaces, as shown below.
\begin{theorem}\label{Teorema0}
	Assume $s\in\bR$, $\cA\in Sp(d,\bR)$. Then the metaplectic operator $\mu(A):\cS(\rd)\to\cS'(\rd)$ extends to a continuous operator on  $M^p_{v_s}(\rd)$, $0<p< \infty$, and for $p=\infty$ it extends to a continuous operator on $\cM^\infty_{v_s}(\rd)$. 
\end{theorem} 
\begin{proof}
	For $1\leq p\leq\infty$ the result follows from \cite[Theorem 6.1.8]{Elena-book}, with weight function  $\mu(z)=v_s(z)$, $s\in\bR$, and observing that $v_s\circ \cA\asymp v_s$ since $\det \cA\not=0$.
	For $0<p<1$ we can use similar  arguments as in the proof of  \cite[Theorem 6.1.8]{Elena-book}.
	Namely, consider the lattice $\Lambda=\alpha\zd\times\beta\zd$ and  two windows $g,\gamma\in\cS(\rd)$ such that the related Gabor frame operator $S_{g,\gamma}:=S_{g,\gamma}^\Lambda$ satisfies $S_{g,\gamma}=I$ on $\lrd$. If we set $g_{m,n}:=M_{\beta n} T_{\a m}g$, it remains to prove that  the matrix operator
	\begin{equation}\label{C6matop}
	\{c_{m,n}\}\longmapsto
	\sum_{m,n\in\mathbb{Z}^d} \langle \mu(\cA)
	g_{m,n},g_{m',n'}\rangle c_{m,n}
	\end{equation}
	is bounded from
	$\ell^{p}_{{v_s}}$
	into $\ell^p_{{v_s}}$.
	This follows from Schur's
	test (cf. \cite[Lemma 6.1.7 (ii)]{Elena-book}) if
	we prove that the kernel 
	\[
	K_{m',n',m,n}:=\langle \mu(\cA)
	g_{m,n},g_{m',n'}\rangle
	v_s(m',n')/v_s(m,n),
	\]
satisfies 
	\begin{equation}\label{C66f}
	K_{m',n',m,n}\in
	\ell^\infty_{m,n}\ell^p_{m',n'}.
	\end{equation}

Since 
	\begin{equation}\label{C65f}
	|\la \mu(\cA) g_{m,n}, g_{m',n'}\ra|\leq C v_{-r} (\cA(m,n)
		-(m',n')),
	\end{equation}
	for every $r\geq 0$, cf. \cite[Proposition 5.3]{CGNRJMPA} we have
	\begin{equation}\label{C68f}
	|K_{m',n',m,n}|\lesssim
	v_{-r}(\cA(m,n)-(m',n'))
	\frac{v_s(m',n')}{v_s(\cA(m,n)-(m',n'))
		v_s(m,n)}.
	\end{equation}
	Now, the last quotient in
	\eqref{C68f} is bounded, so
	we deduce $\eqref{C66f}$.
\end{proof}

\begin{corollary}\label{Cor3.14}
	Under the assumptions of Theorem \ref{Teorema0} we have
	\begin{equation}\label{(2)}
	\|\mu({\cA})f\|_{M^{p}_{v_{s}}}\asymp \|f\|_{M^{p}_{v_{s}}},\quad f\in M^{p}_{v_{s}}(\rd),
	\end{equation}
	(with $\cM^{\infty} (\rd)$ in place of $M^{\infty} (\rd)$ for $p=\infty$).	
\end{corollary}
\begin{proof}
	Using the invertibility property of metaplectic operators,
	$$ \|f\|_{M^{p}_{v_{s}}}=\|\mu(\cA)^{-1}\mu(\cA)f\|_{M^{p}_{v_{s}}}=\|\mu(\cA^{-1})\mu(\cA)f\|_{M^{p}_{v_{s}}}\lesssim \|\mu(\cA)f\|_{M^{p}_{v_{s}}}$$
	where the last estimate follows from Theorem \ref{Teorema0} since $\cA^{-1}\in Sp(d,\bR)$. The reverse inequality is stated in Theorem  \ref{Teorema0}.
\end{proof}
\begin{theorem}\label{Teorema1}
	Assume $f,g \in M^{p}_{v_{s}}(\rd)$,  $0<p\leq \infty$, $s\geq 0$. For any $\cA\in Sp(2d,\bR)$ the $\cA$-Wigner  $W_{\cA}(f,g)$ is in $M^{p}_{v_{s}}(\rdd)$, with
	\begin{equation}\label{stimaT1}
	\|W_{\cA}(f,g)\|_{M^{p}_{v_{s}}}\lesssim \|f\|_{M^{p}}\|g\|_{M^{p}_{v_{s}}}+\|g\|_{M^{p}}\|f\|_{M^{p}_{v_{s}}}.
	\end{equation} 
\end{theorem}
\begin{proof}
	By Theorem \ref{Teorema0} (with dimension $2d$ in place of $d$) we can write
	\begin{equation}\label{ET1}
\|	W_{\cA}(f,g)\|_{M^{p}_{v_{s}}}=\|\mu(\cA)(f\otimes \bar{ g})\|_{M^{p}_{v_{s}}}\lesssim \|f\otimes \bar{ g}\|_{M^{p}_{v_{s}}}.
	\end{equation}
	Note also that $v_{s}(z,\zeta)\asymp (v_{s}\otimes 1)(z,\zeta)+(1\otimes v_{s})(z,\zeta)$, so that 
	\begin{align*}
 \|	W_{\cA}(f,g)\|_{M^{p}_{v_{s}}}&\lesssim \|f\otimes \bar{ g}\|_{M^{p}_{v_{s}\otimes 1+1\otimes v_{s}}}\\
 &\lesssim \|f\otimes \bar{ g}\|_{M^{p}_{v_{s}\otimes 1}}+ \|f\otimes \bar{ g}\|_{M^{p}_{1\otimes v_{s}}}\\
 &=\|f\|_{M^{p}_{v_{s}}} \|{ g}\|_{M^{p}}+\|f\|_{M^{p}} \|{ g}\|_{M^{p}_{v_{s}}}.
	\end{align*}
	The proof is concluded.
\end{proof}

\begin{theorem}\label{Teorema2}
	Assume $f \in M^{p}_{v_{s}}(\rd)$,  $0<p\leq 2$, $s\geq 0$, $\cA\in Sp(2d,\bR)$. Then  the following statements are equivalent:\\
	$(i)$  $f\in M^p_{v_{s}}(\rd)$\\
	$(ii)$ $W_{\cA}(f)\in M^p_{v_{s}}(\rdd)$.
\end{theorem}
\begin{proof}
	If $f(t)=0$ for a.e. $t$ then $W_{\cA}(f)\phas=0$ and the equivalence  is trivially  true. Let us now consider the non-trivial case. \par\noindent
	$(i)\Rightarrow (ii)$. It is a consequence of Theorem \ref{Teorema1}. In particular, from \eqref{stimaT1} for $f=g$ we have
	$$\|W_{\cA}(f)\|_{M^{p}_{v_{s}}}\lesssim \|f\|_{M^{p}_{v_{s}}} \|{f}\|_{M^{p}}\lesssim \|f\|^2_{M^{p}_{v_{s}}}.$$
		$(ii)\Rightarrow  (i)$. Fixing $f=g$ and using \eqref{(2)},
		$$\|W_{\cA}(f)\|_{M^{p}_{v_{s}}}=\|\mu(\cA)(f\otimes\bar{ f})\|_{M^{p}_{v_{s}}}\asymp\|f\otimes\bar{ f}\|_{M^{p}_{v_{s}}}.$$
		Note that
		$$\|f\otimes\bar{ f}\|_{M^{p}_{v_{s}\otimes1}}\asymp \|f\|_{M^{p}_{v_{s}}}\|f\|_{M^{p}}.$$
		So, for $f\in L^2(\rd)\setminus\{0\}$, we have
			$$ \|f\|_{M^{p}_{v_{s}}}\asymp \frac1{\|f\|_{M^{p}}}\|f\otimes\bar{ f}\|_{M^{p}_{v_{s}\otimes1}}\lesssim \frac1{\|f\|_{L^2}}\|f\otimes\bar{ f}\|_{M^{p}_{v_{s}}},$$
			since 
			$\|f\|_{L^2}\lesssim \|f\|_{M^p}$, $0<p\leq 2$.
\end{proof}
\begin{theorem}[Inversion formula for the $\cA$-Wigner distribution] 
	Consider $g_1,g_2\in L^2(\rd)$ with $\la g_1,g_2\ra\not=0$, $\cA\in Sp(2d,\bR)$. Then, for any $f\in\lrd$, 
	\begin{equation}\label{invAW}
	f=\frac{1}{\la g_2,g_1\ra}\intrd \mu(\cA^{-1}) W_{\cA} (f,g_1)\phas  g_2\,d\xi.
	\end{equation}
\end{theorem}
\begin{proof}
	Observing that
	$$ \mu(\cA^{-1}) W_{\cA} (f,g_1)= \mu(\cA^{-1}) \mu(\cA) (f\otimes\bar{ g_1})=f\otimes\bar{ g_1},$$
	we can write
	$$\intrd \mu(\cA^{-1}) W_{\cA} (f,g_1)(x,\xi) g_2(\xi)\,d\xi=\intrd f(x)\bar{ g_1}(\xi)g_2(\xi)\,d\xi=f(x)\la g_2,g_1\ra$$
	and the equality \eqref{invAW} follows.
\end{proof}
%If we introduce   the matrix $\mathcal{A}_{\tau}'\in\Spnr$
%\begin{equation}\label{A'}
%\mathcal{A}_{\tau}'=\left(\begin{array}{cc}
%0_{d\times d} & \left(1-\tau\right)I_{d\times d}\\
%-\tau I_{d\times d} & 0_{d\times d}
%\end{array}\right),
%\end{equation}
% properties  \cite[Lemma 2.2]{CNT}
%\begin{lemma}\label{atau} For $\tau\in [0,1]$ we have
%	$$\left(\mathcal{A}_{\tau}'\right)^{\top}=-\mathcal{A}_{1-\tau}'$$
%	 and  $$\mathcal{A}_{\tau}'+\mathcal{A}_{1-\tau}'=J.
%	$$ 
%\end{lemma}
\begin{proposition}\label{Prop3}
	For $f,g_1,g_2,g_3\in \lrd$, $\cA\in Sp(2d,\bR)$, we have
	\begin{equation}
	V_{g_3} f (w)= \frac{1}{\la g_2, g_1 \ra}\la W_{\cA}(f,g_1),W_{\cA}(\pi(w)g_3,g_2)\ra_{L^2(\rdd)}.
	\end{equation}
\end{proposition}
\begin{proof}
From the preceding inversion formula \eqref{invAW} we have
\begin{align*}
V_{g_3} f (w)&=\frac{1}{\la g_2, g_1 \ra}\intrdd  W_{\cA}(f,g_1)W_{\cA}(\pi(w)g_1,g_2)\overline{\pi(w)g_3(x)} dx d\xi\\
&=\frac{1}{\la g_2, g_1 \ra}\intrdd W_{\cA}(\pi(w)g_1,g_2)(x,\xi)\overline{\mu(\cA)(\overline{g_2(\xi)}\pi(w)g_3(x))}dx d\xi,
\end{align*}
since $\mu(\cA^{-1})=\mu(\cA)^*$. Observe that the integrals above are absolutely convergent integrals since $\pi(w)$ is an isometry on $\lrd$ and $W_\cA: L^2(\rd)\times L^2(\rd)\to L^2(\rdd)$, by Proposition \ref{def bilA triple}. This concludes the proof.
\end{proof}

Proposition \ref{Prop3}  suggests the following definition:
\begin{definition}\label{semi-covariant}
Given $\cA\in Sp(2d,\bR)$, we say that $W_{\cA}$ is \emph{\bf shift-invertible}  if
	$$ |W_{\cA}(\pi(w)f,g)|=|T_{E_\cA(w)}W_{\cA}(f,g)|,\quad f,g\in\lrd,\quad w\in\rdd,$$
	for some $E_\cA\in GL(2d,\bR)$, with $$T_{E_\cA(w)}W_{\cA}(f,g)(z)=W_{\cA}(f,g)(z-E_\cA w),\quad w,z\in\rdd.$$
\end{definition}

Not every $\cA$-Wigner satisfies the above property. Let us  compute $W_{\cA}(\pi(w)f,g)$ explicitly. Consider $\cA$ with the block decomposition in \eqref{blockmatrixA}, and sub-bocks \eqref{sub-block}. Easy calculations  and the intertwining formula $\mu(\cA)\pi(z)=c_{\cA}\pi(\cA z)\mu(\cA)$ with $|c_{\A}|=1$ show, for $w=(w_1,w_2)$,
\begin{align*}
W_{\cA}&(\pi(w)f,g)=\mu(\cA)((\pi(w)f)\otimes\bar{ g})\\
&=\mu(\cA)\pi(w_1,0,w_2,0)(f\otimes\bar{g})\notag\\
&=c_{\cA} \pi (\cA (w_1,0,w_2,0)^T)W_{\cA}(f,g),\notag\\
&=c_{\cA}\pi(A(w_1,0)^T+B(w_2,0)^T, C(w_1,0)^T+D(w_2,0)^T)W_{\cA}(f,g)\notag\\
&=c_{\cA}\pi(A_{11}w_1,A_{21}w_1)+(B_{11}w_2,B_{21}w_2), (C_{11}w_1,C_{21}w_1)+(D_{11}w_2,D_{21}w_2))W_{\cA}(f,g)\notag\\
&=c_{\cA}\pi(A_{11}w_1+B_{11}w_2,A_{21}w_1+B_{21}w_2, C_{11}w_1+D_{11}w_2,C_{21}w_1+D_{21}w_2)W_{\cA}(f,g)\notag\\
&=c_{\cA}M_{C_{11}w_1+D_{11}w_2,C_{21}w_1+D_{21}w_2}T_{A_{11}w_1+B_{11}w_2,A_{21}w_1+B_{21}w_2}W_{\cA}(f,g).\notag
\end{align*}
so that 
$$|W_{\cA}(\pi(w)f,g)|=|T_{A_{11}w_1+B_{11}w_2,A_{21}w_1+B_{21}w_2}W_{\cA}(f,g)|.$$
Hence the matrix $E_\cA$ in Definition \ref{semi-covariant} is given by
	\begin{equation}\label{matrixE}
E_\cA=\left(\begin{array}{cc}
A_{11} & B_{11}\\
A_{21} & B_{21}
\end{array}\right).
\end{equation}
$W_{\cA}$ is shift-invertible if and only if the matrix $E_\cA$ is invertible.\par

\begin{remark}\label{rem2.20}
(i) If  $\cA\in Sp(2d,\bR)$ is a covariant matrix  then
\begin{equation}\label{E-covariant}
E_\cA=\left(\begin{array}{cc}
	A_{11} & 	A_{13}\\
	A_{21} & I_{d\times d}-	A_{11}^T
\end{array}\right).
\end{equation}
Hence if  $E_\cA$  is invertible the  covariant matrix $\cA$ is  shift-invertible.\par\noindent
(ii) For $\tau$-Wigner distributions the matrix $\cA={\bf A}_{\tau}$ is shown in \eqref{Aw-tau}.
The related matrix  $E_\tau:=E_{{\bf A}_{\tau}}$ is
$$
E_\tau=\left(\begin{array}{cc}
(1-\tau) I_{d\times d} & 0_{d\times d}\\
0_{d\times d} & \tau I_{d\times d}
\end{array}\right),
$$
so  that $ {\bf A}_{\tau}$ is covariant for every $\tau\in\bR$, whereas  it is shift-invertible for $\tau\in \bR\setminus\{0,1\}$.\par\noindent
(iii) For $f,g\in\lrd$, $V_g f=\mu({\bf A_{ST}})(f\otimes \bar{g})$, 
and we have, c.f. \cite[Proposition 1.2.15]{Elena-book}, 
\begin{equation}\label{E1}
|V_g \pi(w)f|=|T_{w}V_gf|,\quad w\in\rdd.
\end{equation}
This implies that $\cA=A_{ST}$ in \eqref{A-FT}
is shift-invertible. Observe that in this case, $E_{ST}:=E_{\bf A_{ST}}$ is
$$ E_{ST}=I_{2d\times2d}=\left(\begin{array}{cc}
 I_{d\times d} & 0_{d\times d}\\
0_{d\times d} &I_{d\times d}
\end{array}\right) 
$$
according to relation in \eqref{E1}. We notice that ${\bf A_{ST}}$ is not covariant.
\end{remark}

\begin{proposition}[Relation between the matrix $E_\cA$  and $B_\cA$]\label{E-B}
 If  $\cA\in Sp(2d,\bR)$ is a covariant matrix with related matrix $E_\cA$ in \eqref{E-covariant} and symmetric matrix  $B_\cA$ in \eqref{aggiunta3}, then
  \begin{equation}\label{E-B-eq}
 E_\cA J+\frac12 J= B_\cA.
 \end{equation}
\end{proposition}
\begin{proof}
It is a simple computation. In fact,
\begin{align*}
E_\cA J+\frac{J}{2}&=\left(\begin{array}{cc}
A_{13} & -A_{11}	\\
I_{d\times d}-	A_{11}^T & -A_{21}
\end{array}\right)+\frac12 \left(\begin{array}{cc}
0_{d\times d} & I_{d\times d}	\\
-I_{d\times d} & 0_{d\times d}
\end{array}\right)\\
&=\left(\begin{array}{cc}
A_{13}  & -A_{11}+\frac12 I_{d\times d}	\\
\frac12 I_{d\times d}-A_{11}^T & -A_{21}
\end{array}\right)={B_\cA}.
\end{align*}
This concludes the proof.
\end{proof}

Observe that the next result extends \cite[Theorem  3.11]{CR2021} to every $0<p\leq \infty$.
\begin{theorem}\label{mainE}
	Fix $g\in \cS(\rd)$. For $\cA\in Sp(2d,\bR)$ we have the following issues:\par
	(i) For $0<p<2$, if $f\in M^p_{v_{s}}(\rd)$ then $W_{\cA}(f,g)\in L^p_{v_{s}}(\rdd)$.\par
	(ii) Let $W_{\cA}$ be  shift-invertible according to the preceding definition. Then,\par
	(iia) For $s\geq0$, $1\leq p\leq2$, 
	\begin{equation}\label{norm-eq-p}
	f\in M^p_{v_{s}}(\rd)\,\,\Leftrightarrow \,\,W_{\cA}(f,g)\in L^p_{v_{s}}(\rdd),
	\end{equation}
	with equivalence of norms  $\|f\|_{M^p_{v_{s}}}\asymp \|W_{\cA}(f,g)\|_{L^p_{v_{s}}}$.\par 
	(iib) For $1\leq p\leq\infty$, if $W_{\cA}(f,g)	\in L^p_{v_{s}}(\rdd)$ then $f\in M^p_{v_{s}}(\rd)$.\par
	(iic) For $0<p<1$, if $W_{\cA}(f,g)	\in L^p_{v_{s}}(\rdd)$ and there exists a Gabor frame $\cG(\gamma,\Lambda)$ for $\lrd$  with $\gamma\in\cS(\rd)$ such that the sequence $W_{\cA}(f,\gamma)(\lambda)\in \ell^p_{v_{s}}(\Lambda)$, then $f\in  M^p_{v_{s}}(\rd)$.\par
\end{theorem}
\begin{proof}
(i) Let us recall that $\cS(\rd)\subset M^p_{v_{s}}(\rd)$, $0<p\leq\infty$, $s\in\bR$. Assume first $f\in M^p_{v_{s}}(\rd)$, $s\geq0$. Then by Theorem \ref{Teorema1} we have
$$\|W_{\cA}(f,g)\|_{M^{p}_{v_{s}}}\lesssim \|f\|_{M^{p}_{v_{s}}}\|g\|_{M^{p}_{v_{s}}}.$$
Since $(v_{s}\otimes 1)(x,\xi)\lesssim v_{s}(x,\xi)$ for $s\geq0$,  the inclusion relations for modulation spaces (cf., e.g., \cite[Theorem 2.4.17]{Elena-book} and \cite[Proposition 1.2]{ToftquasiBanach2017}) yield
$$M^{p}_{v_{s}}(\rdd) \hookrightarrow M^{p}_{v_{s}\otimes 1}(\rdd),$$
For $0< p\leq 2$ (see \cite[Proposition 2.9]{Toftweight2004}, whereas the case $0<p<1$ is a direct consequence of  \cite[Theorem 2.4]{ToftACHA19} with $\mathcal{B}=L^p_{v_{s}}$)
$$M^{p}_{v_{s}\otimes 1}(\rdd) \hookrightarrow L^{p}_{v_{s}}(\rdd),$$
hence $W_{\cA}(f,g)\in L^{p}_{v_{s}}(\rdd)$.\par
(ii) Assume now that $W_\cA$ is shift-invertible and $W_\cA (f,g)\in L^p_{v_{s}}(\rdd)$.  Then, by Proposition \ref{Prop3}, with $g_1=g_3$, 
\begin{align*}
	|V_{g_1} f (w)|&\lesssim\frac{1}{|\la g_2, g_1 \ra|}|\la W_{\cA}(f,g_1),W_{\cA}(\pi(w)g_1,g_2)\ra_{L^2(\rdd)}|\\
	&\lesssim \intrdd |W_{\cA}(f,g_1)|(u)|W_{\cA}(\pi(w)g_1,g_2)|(u)du\\
	&\lesssim \intrdd |W_{\cA}(f,g_1)|(u) |W_{\cA}(g_1,g_2)|(u-E_\cA w)du\\
	&\lesssim \intrdd |W_{\cA}(f,g_1)|(u) |[W_{\cA}(g_1,g_2)]^*|(E_\cA w-u)du
\end{align*}
Hence
\begin{align*}\|f\|_{M^p_{v_{s}}}&\asymp \|V_{g_1} f \|_{L^p_{v_{s}}}\lesssim \| |W_\cA(f,g_1)|\ast  |[W_{\cA}(g_1,g_2)]^*|(E_\cA\cdot)\|_{L^p_{v_{s}}}\\
&\asymp  \| |W_\cA(f,g_1)|\ast  |[W_{\cA}(g_1,g_2)]^*\|_{L^p_{v_{s}}}
\end{align*}
since $v_{s}(y)\asymp v_{s}(E_\cA^{-1}y)$.  Now, Young's convolution inequalities for $1\leq p\leq \infty$ gives
$$  \| |W_\cA(f,g_1)|\ast  |[W_{\cA}(g_1,g_2)]^*\|_{L^p_{v_{s}}}\leq  \|W_\cA(f,g_1)\|_{L^p_{v_{s}}}\|W_{\cA}(g_1,g_2)\|_{L^1_{v_{s}}}<\infty,$$
since $W_{\cA}(g_1,g_2)\in\cS(\rdd)$ for $g_1,g_2\in\cS(\rd)$ by Proposition \ref{def bilA triple}.
This proves the implication in $(iib)$. Moreover, item $(i)$ and the previous estimate yield the equivalence in $(iia)$. It remains to show item $(iic)$. 
For $0<p<1$,  consider $\gamma\in \cS(\rd)$ such that  $G(\gamma; \Lambda)$ is a Gabor frame for $L^2(\rd)$, then, arguing as above with $\gamma$ in place of $g_3$:
\begin{align*}\|f\|_{M^p_{v_{s}}}&\asymp \|V_{\gamma} f \|_{L^p_{v_{s}}}\asymp \|V_{\gamma} f \|_{\ell^p_{v_{s}}}\lesssim \| |W_\cA(f,\gamma)|\ast  |[W_{\cA}(\gamma,g_2)]^*|(E\cdot)\|_{\ell^p_{v_{s}}}\\
&\asymp  \| |W_\cA(f,\gamma)|\ast  |[W_{\cA}(\gamma,g_2)]^*\|_{\ell^p_{v_{s}}}\leq \|W_\cA(f,\gamma)\|_{\ell^p_{v_{s}}}\|[W_{\cA}(\gamma,g_2)]^*\|_{\ell^p_{v_{s}}}<\infty,
\end{align*}
 by the convolution property for sequences: $$\ell^p_{v_{s}}\ast \ell^p_{v_{s}}\hookrightarrow \ell^p_{v_{s}},\quad s\geq0,\,\,0<p\leq1;$$
 and $W_{\cA}(\gamma,g_2)\in\cS(\rdd)$ for $\gamma,g_2\in\cS(\rd)$ (cf., Proposition \ref{def bilA triple}). In fact,  the restriction $W_{\cA}(\gamma,g_2)(\lambda)$, $\lambda\in\Lambda$ is in $\ell^p_{v_{s}}(\Lambda)$, for every $0<p\leq\infty$, $s\geq0$. This concludes the proof.
\end{proof}

\begin{corollary}\label{mainEcor}
	For $s\geq0$, $1\leq p\leq2$,  $\cA\in Sp(2d,\bR)$ such that $W_\cA$ is shift-invertible. Then
	\begin{equation*}
	f\in M^p_{v_{s}}(\rd)\,\,\Leftrightarrow \,\,W_{\cA}f\in L^p_{v_{s}}(\rdd).
	\end{equation*}	
\end{corollary}
\begin{proof}
	$f\in M^p_{v_{s}}(\rd)\,\,\Rightarrow \,\,W_{\cA}f\in L^p_{v_{s}}(\rdd)$  is a straightforward generalization of the proof of Theorem \ref{mainE}$(i)$, with $g=f\in M^p_{v_{s}}(\rd)$. Vice versa, following the proof  pattern of  Theorem \ref{mainE} $(ii)$ with Proposition \ref{Prop3} applied for $g_1,g_2,g_3\in\cS(\rd)$ we can write
\begin{equation}\label{App3}
	\|f\|_{M^p_{v_{s}}}\lesssim \|W_{\cA}(f,g_1)\|_{L^p_{v_s}}\|W_\cA(g_3,g_2)\|_{L^1_{v_s}}.
\end{equation}
	Now, for $f\in  M^p_{v_{s}}(\rd)$ there exists a sequence $(g_n)_n \subset \cS(\rd)$ such that $g_n\to f$ in $ M^p_{v_{s}}(\rd)$. Now, using \cite[Proposition 2.9]{Toftweight2004} in the first inequality below  and \cite[Proposition 2.4.17]{Elena-book} in the second one, for $1\leq p\leq 2$, 
	\begin{align*}
	\|W_{\cA}(f,f)-W_{\cA}(f,g_n)\|_{L^p_{v_s}} &=\|W_{\cA}(f,f-g_n)\|_{L^p_{v_s}} \leq \|W_{\cA}(f,f-g_n)\|_{M^p_{v_{s}\otimes 1}}\\&\leq \|W_{\cA}(f,f-g_n)\|_{M^p_{v_s}} =\|W_{\cA}(f,f-g_n)\|_{M^p_{v_s}}\\ &\lesssim \|f\|_{M^{p}_{v_{s}}}\|g_n-g\|_{M^{p}_{v_{s}}},
	\end{align*}
	where the last inequality is due to Theorem \ref{Teorema1}. Since $\|g_n-g\|_{M^{p}_{v_{s}}}\to 0$ as $n\to\infty$, we obtain that $\|W_{\cA}(f,g_n)\|_{L^p_{v_s}} \to \|W_{\cA}(f,f)\|_{L^p_{v_s}}$ as $n\to\infty$ and the thesis follows by replacing $g_1$ by $g_n$ in \eqref{App3} and letting $n\to\infty$.	
\end{proof}

For  $\tau$-Wigner distributions we have a characterization for every $0<p\leq\infty$, as explained below. Notice that we extend Theorem 3.11 of \cite{CR2021} to $0<p\leq\infty$ for the weight $v=v_{s}$, $s\geq0$.
\begin{proposition}\label{caso-tau}
		Consider $0< p,q\leq \infty$,  $\tau\in \bR\setminus\{(0,1)\}$. Then, for any $g\in\cS(\rd)$,
		\begin{equation}\label{Mpvgammanew}
		f\in M^{p,q}_{v_{s}}(\rd)\Leftrightarrow W_\tau (f,g) \in L^{p,q}_{v_{s}}(\rdd).
		\end{equation}  
		For $1\leq p,q\leq\infty$ the window $g$ can be chosen in the larger class $M^1_{v_{s}}(\rd)$.
	\end{proposition}
\begin{proof}
	For $p=q$ and $1\leq p\leq\infty$ the result was proved in Theorem 3.11 of \cite{CR2021}. Let us prove the general case. By Corollary 3.3. of \cite{CR2021}, with $Q_\tau g$ in place of $g$, we can write
	\begin{equation*}
	V_{Q_{\tau}g} f (x,\xi)=\tau^d e^{-2\pi i(1- \tau)x \xi}	W_\tau (f,g)\left(\mathcal{B}^{-1}_\tau\phas\right),
	\end{equation*}
	where
	\begin{equation*}
	Q_\tau g(t)=\mathcal{I}g\left(\frac{1-\tau}{\tau}t\right),\quad t\in\rd,
	\end{equation*}
	$\mathcal{I}g(t):=g(-t)$, and 
		\begin{equation*}
\mathcal{B}_\tau^{-1}=\begin{pmatrix}(1-\tau) I_d&0_d\\0_d&\tau I_d\end{pmatrix}.
	\end{equation*}
	The result is then a simple computation:
	\begin{align*}
	\|f\|_{M^{p,q}_{v_{s}}}&\asymp \|V_{Q_{\tau}g} f\|_{L^{p,q}_{v_{s}}}=\|W_\tau (f,g)(\mathcal{B}^{-1}_\tau\cdot)\|_{L^{p,q}_{v_{s}}}\asymp\|W_\tau (f,g)((1-\tau)\cdot,\tau\cdot)\|_{L^{p,q}_{v_{s}}}\\
	&\asymp  \|W_\tau (f,g)\|_{L^{p,q}_{v_{s}}},
	\end{align*}
	since $v_{s}((1-\tau)\cdot,\tau\cdot))\asymp_\tau v_{s}$, for $\tau=\bR\setminus\{0,1\}$.
\end{proof}

\subsection{STFT and $\cA$-Wigner representations}

The case of $\tau$-Wigner distributions suggests a deeper study of covariant matrices $\cA$ such that $$\mu(\cA)=\cF_2 \mathfrak{T}_{L}$$ as in \eqref{metapf2dil},
where $\cF_2$ is the partial Fourier transform with respect to the second variables $y$ defined in \eqref{FT2} and the change of coordinates $\mathfrak{T}_{L}$ is defined in \eqref{Ltransf}. As observed in \cite{CR2021},  see also \cite{MO2002},
\begin{equation}\label{MOf2}
\mu(\cA_{FT2})=\cF_2,
\end{equation}
where \begin{equation}\label{A-f2}
\mathcal{A}_{FT2}=\left(\begin{array}{cc}
A_{11}^{FT2} & A_{12}^{FT2}\\
A_{21}^{FT2} & A_{22}^{FT2}
\end{array}\right)\in Sp(2d,\bR),
\end{equation}
and  $A_{11}^{FT2},A_{12}^{FT2},A_{21}^{FT2},A_{22}^{FT2}$ are the $2d\times 2d$ matrices: 
\begin{equation}\label{A-f2bis}
{A}_{11}^{FT2}=A_{22}^{FT2}=\left(\begin{array}{cc}
I_{d\times d} &0_{d\times d}\\
0_{d\times d} &0_{d\times d}
\end{array}\right), \quad {A}_{12}^{FT2}=\left(\begin{array}{cc}
0_{d\times d} &0_{d\times d}\\
0_{d\times d} &I_{d\times d}
\end{array}\right), \quad {A}_{21}^{FT2}=-{A}_{12}^{FT2}.
\end{equation}

\begin{proposition}\label{E7} A covariant matrix $\cA\in Sp(2d,\bR)$ satisfies \eqref{metapf2dil} if and only if
	\begin{equation}\label{A-covariant-L}
	\cA=\left(\begin{array}{cccc}
	A_{11} & I_{d\times d}-A_{11}&0_{d\times d}&0_{d\times d}\\
0_{d\times d}&0_{d\times d}&I_{d\times d}- A^T_{11} &-A^T_{11}\\
	0_{d\times d}&0_{d\times d}&I_{d\times d} &I_{d\times d}\\
	-I_{d\times d}&I_{d\times d}& 0_{d\times d} &0_{d\times d}\\
	\end{array}\right)
	\end{equation}
(observe that $A_{13}= A_{21}=0_{d\times d}$)	and the matrix $L$ in \eqref{Ltransf} is given by 
	\begin{equation}\label{L-cov}
	L=\left(\begin{array}{cc}
	I_{d\times d} & I_{d\times d}-A_{11}\\
	I_{d\times d}&- A_{11}\\
	\end{array}\right).
	\end{equation}
\end{proposition}
\begin{proof}
Up to a phase factor we can write
$$\cA=\mathcal{A}_{FT2}\cD_L,$$
where $\cD_L$ is defined in \eqref{MotL}. The claim is then a straightforward computation, using that
	\begin{equation*}
L^{-1 }=\left(\begin{array}{cc}
 A_{11} & I_{d\times d}-A_{11}\\
I_{d\times d}&-I_{d\times d}\\
\end{array}\right).
\end{equation*}
\end{proof}
\begin{remark}
	(i) The matrix $L$ in \eqref{L-cov} is invertible for every $d\times d$ real matrix $A_{11}$. (We stress that $A_{11}$ is not required to be invertible). In fact, we have
	$$\det L =\det(-I_{d\times d})=(-1)^d.$$
	(ii) Under the assumptions of Proposition \ref{E7} the matrix $E_\cA$ becomes 
		\begin{equation}\label{matrixE-L}
	E_\cA=\left(\begin{array}{cc}
	A_{11} & 0_{d\times d}\\
		0_{d\times d} &  I_{d\times d}-A_{11}^T
	\end{array}\right)
	\end{equation}
	so that $E_\cA$ is invertible if and only if $A_{11}$ and $I_{d\times d}-A_{11}^T$ (or, equivalently, $I_{d\times d}-A_{11}$) are invertible matrices. In other words, $\cA$ is shift-invertible if and only if $A_{11}$ and $I_{d\times d}-A_{11}$ are invertible matrices.\par
	(iii) For $\tau$-Wigner distributions the matrix $L=L_{\tau}$ is easily computed to be
	\begin{equation}\label{matrixE-L-tau}
	L_{\tau}=\left(\begin{array}{cc}
	I_{d\times d} & \tau I_{d\times d}\\
	I_{d\times d} &  -(1-\tau)I_{d\times d}
	\end{array}\right).
	\end{equation}
	\end{remark}

We are interested to determine the conditions under which a covariant $\cA$-Wigner $W_\cA=c_\cA\cF_2 \mathcal{D}_L$  with $|c_\cA|=1$, can be related to the STFT. We recall that the matrix $L$ takes the form in \eqref{L-cov} so that 
\begin{equation*}
(L^{-1})^T=\left(\begin{array}{cc}
A_{11}^T & I_{d\times d}\\
 I_{d\times d}-A_{11}^T & -I_{d\times d}
\end{array}\right).
\end{equation*}

\begin{theorem}
	\label{right-reg rep} Let $\cA\in Sp(2d,\bR)$ be a covariant matrix satisfying \eqref{metapf2dil}  and shift-invertible. For every $f,g\in L^{2}\left(\mathbb{R}^{d}\right)$, the following formula holds:
	\begin{equation}\label{a-w-STFT}
	W_{\cA}\left(f,g\right)\left(x,\o\right)=\left|\det (I_{d\times d}-A_{11})\right|^{-1}e^{2\pi i (I-A_{11}^T)^{-1}\o\cdot x}V_{\tilde{g}}f(A_{11}^{-1}x,(I-A_{11}^T)^{-1}\o),\quad x,\o\in\rd,
	\end{equation}
	where 
	\begin{equation}\label{tildeg}
	\tilde{g}\left(t\right)=g\left(-A_{11}(I_{d\times d}-A_{11})^{-1}t\right).
	\end{equation}
\end{theorem}
\begin{proof}
	Since $\cA$ is shif-invertible the matrices  $A_{11}$ and $I_{d\times d}-A_{11}$ are invertible. Then the result follows from  Theorem 3.8 of \cite{CT2020}.
\end{proof}
\begin{theorem}\label{caso-A-Wigner}
	Consider $0< p,q\leq \infty$,  $\cA\in Sp(2d,\bR)$ as in Theorem \ref{right-reg rep}. Then, for any $g\in\cS(\rd)$,
	\begin{equation}\label{Mpvgammanew2}
	f\in M^{p,q}_{v_{s}}(\rd)\Leftrightarrow W_\cA (f,g) \in L^{p,q}_{v_{s}}(\rdd).
	\end{equation}  
	 with equivalence of norms $\|f\|_{M^{p,q}_{v_{s}}}\asymp \|W_\cA (f,g) \|_{L^{p,q}_{v_{s}}}$.
	For $1\leq p,q\leq\infty$ the window $g$ can be chosen in the larger class $M^1_{v_{s}}(\rd)$.
\end{theorem}
\begin{proof}
	It is a straightforward consequence of Theorem \ref{right-reg rep}. In fact, for $g\in\cS(\rd)$ and under the assumptions $\det A_{11}\not=0$, $\det(I_{d\times d}-A_{11}) \not=0$, the rescaled function $\tilde{g}$ in \eqref{tildeg} is in $\cS(\rd)$ and by \eqref{a-w-STFT},
	\begin{equation*}
	\|f	\|_{M^{p,q}_{v_{s}}}\asymp \|V_{\tilde{g}}f\|_{ L^{p,q}_{v_{s}}}\asymp \|W_\cA (f,g)( A_{11}\cdot,\cdot)\|_{ L^{p,q}_{v_{s}}}
	\asymp \|	W_{\cA}\left(f,g\right)\|_{ L^{p,q}_{v_{s}}},
	\end{equation*}
	since $$v_{s}(A_{11}^{-1}z_1,z_2)=(1+|A_{11}^{-1}z_1|^2+|z_2|^2)^{s/2}\asymp (1+|z_1|^2+|z_2|^2)^{s/2},\quad s\in\bR.$$
	For $p,q\geq 1$ the windows can be chosen in the larger class $M^1_{v_{s}}(\rd)$ and we can argue as above by observing that $\tilde{g}$ in \eqref{tildeg} is in $M^1_{v_{s}}(\rd)$ whenever $g$ is.
\end{proof}

\section{$\cA$-Perturbations of the Wigner distribution}

This section studies the covariant $\cA$-Wigner representations as \emph{perturbations} of the Wigner distributions in \eqref{aggiunta6}:
	\begin{equation*}
W_\cA (f,g)=W(f,g)\ast\sigma_\cA\quad f,g\in\cS(\rd),
\end{equation*}
where the kernel $\sigma_\cA$ is defined in \eqref{aggiunta7}.
 We revisit in wider generality the  linear perturbations already performed in \cite{CT2020}. First, we recall the expression of the kernel $\sigma_\cA$ from Theorem \ref{Thaggiunta1}:
\begin{corollary}\label{CorThaggiunta1}
For a covariant matrix $\cA\in Sp(2d,\bR)$  satisfying \eqref{metapf2dil} the matrix $B_\cA$ in \eqref{aggiunta3} becomes
\begin{equation}\label{aggiunta3-FL}
B_\cA=\left(\begin{array}{cc}
0_{d\times d} & \frac12 I_{d\times d} -A_{11}\\
\frac12 I_{d\times d} -A_{11}^T&0_{d\times d}
\end{array}\right),
\end{equation}
so that the kernel $\sigma_\cA$ can be rephrased as
\begin{equation}\label{aggiunta7-FL}
\sigma_\cA(z)=\cF^{-1}(e^{-\pi i \zeta\cdot B_\cA\zeta})(z)=\cF^{-1}(e^{-\pi i \zeta_1\cdot \zeta_2}e^{-2\pi i \zeta_1\cdot A_{11}\zeta_2})(z).
\end{equation}
In particular, if $(1/2) I_{d\times d} -A_{11}$ is invertible, then by  \cite[Theorem 4.7]{CR2021}
\begin{align}\label{aggiunta8}
\sigma_\cA(z)&= e^{\pi i \sharp (B_\cA)}|\det B_\cA| e^{-\pi i z\cdot B_\cA^{-1}z}\\
&= e^{\pi i \sharp (B_\cA)}(\det ((1/2) I_{d\times d} -A_{11}))^2 e^{-\pi i z_1\cdot ( \frac12 I_{d\times d} -A_{11}^T)^{-1}z_2},\notag
\end{align}
where $ \sharp (B_\cA)$ is the number of positive eigenvalues of $B_\cA$ minus the number of negative eigenvalues and 
\begin{equation}\label{aggiunta3-FL-inv}
B_\cA^{-1}=\left(\begin{array}{cc}
0_{d\times d} &( \frac12 I_{d\times d} -A_{11}^T)^{-1}\\
(\frac12 I_{d\times d} -A_{11})^{-1}&0_{d\times d}
\end{array}\right).
\end{equation}
\end{corollary}

We observe that a sufficient condition for  and $(1/2) I-A_{11}^T$ to be invertible is $\|A_{11}\|<1/2$, then $(1/2) I_{d\times d} -A_{11}^T$ is invertible and 
$$((1/2) I_{d\times d} -A_{11}^T)^{-1}=2 (I_{d\times d} -2A_{11}^T)^{-1}=2\sum_{n=0}^{+\infty} (2A_{11}^T)^n.$$
For $\tau\in(0,1)$, $A_{11}^T=A_{11}=(1-\tau)I_{d\times d}$ and the Neumann series gives $((1/2)I_{d\times d} -A_{11}^T)^{-1}=\frac1{\tau-\frac12}I_{d\times d}$,  expected. \par
In what follows we give a precise estimate of the time-frequency content of the chirp function $\Theta(z_1,z_2)=e^{2\pi iz_1\cdot z_2}$, improving \cite[Proposition 3.2 and Corollary 3.4]{cdgn red int} (see also \cite[Proposition 4.7.15]{Elena-book}). 
\begin{lemma}
	\label{chirp spaces} For any $0<p\leq\infty$ the function $\Theta(z_1,z_2)=e^{2\pi iz_1\cdot z_2}$ satisfies  
	$$  \Theta\in M^{p,\infty}_{ v_{s}\otimes 1}\left(\rdd\right)\cap W(\mathcal{F}L^{p}_{v_{s}},L^{\infty})\left(\rdd\right), \quad s\geq0.$$ 
\end{lemma}
\begin{proof}
	We first compute $W(\mathcal{F}L^{p}_{v_{s}},L^{\infty})$-norm of $\Theta$. Proceeding as in the proof of  \cite[Proposition 3.2]{cdgn red int},
	$$\|\Theta\|_{W(\mathcal{F}L^{p}_{v_{s}},L^\infty)(\rdd)}=\sup_{u\in\rdd}\|\cF( \Theta T_u g)\|_{L^p_{v_{s}}(\rdd)}. $$
	Using the Gaussian window $g(\zeta_1,\zeta_2)=e^{-\pi \zeta_1^2} e^{-\pi \zeta_2^2}$ and following the pattern of \cite[Proposition 3.2]{cdgn red int}  we obtain  
	$$\|\cF( \Theta T_u g)\|_{L^p_{v_{s}}}=2^{-d/2}  \|e^{-\frac\pi 2 |\cdot|^2}\|_{L^p_{v_{s}}}=C_{p,s}>0, \quad s\in\bR.$$
	Hence $\|\Theta\|_{W(\mathcal{F}L^{p}_{v_{s}},L^\infty)(\rdd)}=C_{q,s}$, for every $s\geq0$.
	Observe that 
	\begin{equation}\label{FTxp}
\cF\Theta( \zeta_1, \zeta_2)=	\cF(e^{2\pi i z_1\cdot z_2})( \zeta_1, \zeta_2)= e^{-2\pi i  \zeta_1\cdot  \zeta_2},
	\end{equation}
	and a direct computation or an inspection of the proof  of \cite[Proposition 3.2]{cdgn red int} shows 
	$$\|\cF( \cF\Theta T_u g)\|_{L^p_{v_{s}}}=2^{-d/2}  \|e^{-\frac\pi 2 |\cdot|^2}\|_{L^q_{v_{s}}}=C_{p,s}>0, \quad s\in\bR.$$
	In other words, the minus sign at the exponent of $\Theta$ does not affect its norm, so that 
	$$\|\Theta\|_{W(\mathcal{F}L^{p}_{v_{s}},L^\infty)(\rdd)}=\|\cF\Theta\|_{W(\mathcal{F}L^{p}_{v_{s}},L^\infty)(\rdd)}.$$
	Finally, using \eqref{W-M},
$$	\|\Theta\|_{M^{p,\infty}_{ v_{s}\otimes 1}}=\|\cF \Theta\|_{W(\mathcal{F}L^{p}_{v_{s}},L^{\infty})}<\infty,$$
	so we are done.
\end{proof}

In what follows we shall use the dilation properties for modulation spaces. Since we are not aware of dilation properties for quasi-Banach modulation spaces, we state the following result, which extends \cite[Proposition 3.1]{cn met rep} to these cases.
\begin{proposition}[Dilation properties for modulation spaces]\label{dilmod} Let $0<p,q\leq\infty$ and  $A\in GL(d,\R)$, $0<p,q\leq\infty$,  $p_1=\min\{p,1\}$, $q_1=\min\{q,1\}$, $\f(t)=e^{-\pi t^2}$. Then, for every $f\in \mpq(\rd)$,
	\begin{equation}\label{pp1}\|f_A\|_{\mpq}\lesssim
	|\det
	A|^{-(1/p-1/q+1)}\|V_{\f_{A^{-1}}} \f\|_{W(L^1,L^{p_1,q_1})}\|f\|_{\mpq}.
	\end{equation}
	In particular, for $p,q\geq1$, $$\|V_{\f_{A^{-1}}} \f\|_{W(L^1,L^{p_1,q_1})}=\|V_{\f_{A^{-1}}} \f\|_{L^1}\asymp (\det(I+A^TA))^{1/2},$$ cf. \cite[Lemma 3.2]{cn met rep}.
\end{proposition}
\begin{proof} The pattern is similar to \cite[Proposition 3.1]{cn met rep}.  By
	a change of variable,  the dilation is transferred from the function $f$ to the window $\f(t)=e^{-\pi t^2}$:
	$$V_\f f_A\phas=|\det A|^{-1} V_{\f_{A^{-1}}} f(A x,(A^*)^{-1}\o).$$
The change of variables $Ax=u$, $(A^*)^{-1}\o=v$ gives
	\begin{align*}
	\|f_A\|_{\mpq}&=|\det A|^{-1}\left(\intrd\left(\intrd | V_{\f_{A^{-1}}}
	f(A x,(A^*)^{-1}\o)|^p\,dx\right)^{q/p}d\o\right)^{1/q} \\
	&=|\det A|^{-(1/p-1/q+1)}\|V_{\f_{A^{-1}}}f\|_{L^{p,q}}.
	\end{align*}
	Changing the window function (see, e.g., \cite[Lemma 1.2.29]{Elena-book}),
	$$|V_{\f_{A^{-1}}}f\phas|\leq \|\f\|_{L^2}^{-2}  (|V_{\f} f|\ast|V_{\f_{A^{-1}}} \f|)\phas.
	$$
So that 
	\begin{align*}
	\|V_{\f_{A^{-1}}}f\|_{L^{p,q}}&= C \||V_{\f} f|\ast|V_{\f_{A^{-1}}} \f|\|_{L^{p,q}}=\|   |V_{\f} f|\ast|V_{\f_{A^{-1}}} \f|\|_{W(L^{p,q},L^{p,q})}\\
	&\leq \|   |V_{\f} f|\ast|V_{\f_{A^{-1}}} \f|\|_{W(L^{\infty},L^{p,q})},
	\end{align*}
	since $L^\infty\subseteq L^{p,q}$, locally. Now \cite[Corollary 3.1]{Galperin2014} with $X=Z=L^\infty$, $Y=L^1$ (so that $L^\infty\ast L^1\subset L^\infty$) gives 
	$$ \|   |V_{\f} f|\ast|V_{\f_{A^{-1}}} \f|\|_{W(L^{\infty},L^{p,q})}\leq C\|V_{\f} f\|_{W(L^\infty,L^{p,q})}\|V_{\f_{A^{-1}}} \f\|_{W(L^1,L^{p_1,q_1})}
	$$
	with $p_1=\min\{p,1\}$, $q_1=\min\{q,1\}$. Finally, by \cite[Lemma 3.2]{Galperin2004},
	$$\|V_{\f} f\|_{W(L^\infty,L^{p,q})}\leq C\|V_{\f} f\|_{L^{p,q}}\asymp\|f\|_{M^{p,q}},$$
		which concludes the proof.
\end{proof}
\begin{proposition}	\label{kernel spaces} Consider $M\in GL(d,\bR)$  and set
	$$\sigma_M(z) =e^{-\pi i z_1\cdot M z_2}.$$
	Then   we have 
	$$\sigma_M
	\in M^{p,\infty}_{v_{s}\otimes 1}\left(\rdd\right)\cap W(\mathcal{F}L^{p}_{v_{s}},L^{\infty})\left(\rdd\right), \quad s\geq0,$$ 
	for every $0<p\leq\infty$.
\end{proposition}
\begin{proof} We highlight the rescaling matrix in $\sigma_M$ as follows
	\[
	\sigma_{M}\left(z_1,z_2\right)=e^{-\pi iz_1\cdot Mz_2}=D_{\tilde{M}}\Theta\left(z_1,z_2\right),
	\]
	where $\Theta(z_1,z_2)=e^{2\pi iz_1\cdot z_2}$ and $D_{\tilde{M}}$ is the dilation operator $D_{\tilde{M}}F\left(t\right)\coloneqq F\left(\tilde{M}t\right)$
	associated with the invertible matrix $\tilde{M}$:
	\[
	\qquad\tilde{M}=\left(\begin{array}{cc}
	-\frac12 I_{d\times d}& 0\\
	0 & M
	\end{array}\right).
	\]
	It is clear that $\tilde{M}$ is invertible if and only if $M$ is. Now, since the mapping $F\mapsto (v_{s}\otimes 1)F$ is an homeomorphism from $M^{p,\infty}_{v_{s}\otimes 1}(\rdd)$ to $M^{p,\infty}(\rdd)$ (cf. \cite[Corollary 2.3]{Toftweight2004} for $p\geq 1$ and \cite{ACT2020} for $p<1$), we can write
	$$\|D_{\tilde{M}}\Theta\|_{M^{p,\infty}_{v_{s}\otimes 1}}\asymp\|(v_{s}\otimes 1)D_{\tilde{M}}\Theta\|_{M^{p,\infty}}\asymp  \|D_{\tilde{M}}\{[D_{\tilde{M}^{-1}}(v_{s}\otimes 1)]\Theta\}\|_{M^{p,\infty}},$$
	where
		\[
	\qquad\tilde{M}^{-1}=\left(\begin{array}{cc}
	-2 I_{d\times d}& 0\\
	0 & M^{-1}
	\end{array}\right).
	\]
	Observe that 
	$$D_{\tilde{M}^{-1}}(v_{s}\otimes 1)(z_1,z_2)=v_{s}(-2z_1)$$
	so that $D_{\tilde{M}^{-1}}(v_{s}\otimes 1)\asymp v_{s}\otimes 1$ and therefore 
	$$\|D_{\tilde{M}}\{[D_{\tilde{M}^{-1}}(v_{s}\otimes 1)]\Theta\}\|_{M^{p,\infty}}\asymp \|D_{\tilde{M}}[(v_{s}\otimes 1)\Theta]\}\|_{M^{p,\infty}}$$
	and the dilation properties of Proposition \ref{dilmod} yield
	\begin{align*}\|D_{\tilde{M}}[(v_{s}\otimes 1)\Theta]\|_{M^{p,\infty}}&\leq C_{p,M}\|(v_{s}\otimes 1)\Theta\|_{M^{p,\infty}}\\
	&\asymp_{p,M} \|\Theta\|_{M^{p,\infty}_{v_{s}\otimes 1}}<\infty,\end{align*}
	by Lemma \ref{chirp spaces},
	which gives $\sigma_M\in M^{p,\infty}_{v_{s}\otimes 1}(\rdd)$. 
	
	 Now,  condition $\det M\not=0$ yields $\cF \sigma_M(\zeta_1,\zeta_2)= C_M e^{-4\pi i \zeta_1\cdot M^{-1}\zeta_2},$ for a suitable $C_M>0$, so that
	$$\sigma_M(z_1,z_2)=C_M\cF^{-1}(e^{-4\pi i \zeta_1\cdot M^{-1}\zeta_2})(z_1,z_2)=C_M\cF (e^{-4\pi i \zeta_1\cdot M^{-1}\zeta_2})(z_1,z_2).$$
	Using the same argument as above we deduce $e^{-4\pi i \zeta_1\cdot M^{-1}\zeta_2}\in M^{p,\infty}_{v_{s}\otimes 1}(\rdd)$ which gives $\sigma_M\in W(\mathcal{F}L^{p}_{v_{s}},L^{\infty})(\rdd)$, since $\cF M^{p,\infty}_{v_{s}\otimes 1} = W(\mathcal{F}L^{p}_{v_{s}},L^{\infty})$  by \eqref{W-M}. This concludes the proof.
\end{proof}
\begin{theorem}\label{charact mod}
	Let $\cA\in Sp(2d,\bR)$ be a covariant matrix as in \eqref{A-covariant-L} with $B_\cA$ as in \eqref{aggiunta3-FL} and $B_\cA$ invertible (equivalently, $(1/2)I_{d\times d}-A_{11}$ invertible). Then,
   for $0< p,q\leq\infty$,  $f\in\mathcal{S}'\left(\mathbb{R}^{d}\right)$, we have 
	\[
	Wf\in M^{p,q}_{{v_{s}\otimes 1}}\left(\rdd\right) \Leftrightarrow W_{\cA }f\in M^{p,q}_{v_{s}\otimes 1}\left(\rdd\right), \quad s\in\bR.
	\]		
\end{theorem}
\begin{proof}
	Assume first $Wf\in M^{p,q}_{{v_{s}\otimes 1}}\left(\rdd\right)$, for some $0< p,q\leq \infty$, $s\in\bR$. Since $W_\cA f=Wf\ast\sigma_\cA$ (cf. \eqref{aggiunta7-FL}), the result follows by the convolution relations for (quasi-)Banach modulation spaces \cite[Proposition 3.1]{BCN20} and Proposition \ref{kernel spaces} by which   $\sigma_\cA \in M^{r,\infty}_{v_s\otimes 1}(\rdd)$ for any $r=\min\{p,1\}$. This gives the convolution relations:
$$M^{p,q}_{{v_{s}\otimes 1}}(\rdd)\ast M^{r,\infty}_{{v_{s}\otimes 1}}(\rdd)\hookrightarrow M^{p,q}_{{v_{s}\otimes 1}}(\rdd),$$
so that $W_\cA \in M^{p,q}_{{v_{s}\otimes 1}}(\rdd)$.
	
	Vice versa, considering the symplectic Fourier transform of the equality in \eqref{aggiunta6}
 with $\sigma_\cA$ in \eqref{aggiunta8},  we obtain 
	$$\cF_\sigma W_{\cA} f=\cF_\sigma\sigma_\cA \cdot Amb\left(f\right),$$
	where the ambiguity function $Amb\left(f\right)$ is defined in \eqref{ambiguity} and
	$\cF_\sigma\sigma_\cA(\zeta)=e^{-\pi i\zeta\cdot B_\cA \zeta}$.
Thus, multiplying both sides of the previous equality by  $e^{\pi i\zeta\cdot B_\cA \zeta}$ and taking the symplectic Fourier transform again, we obtain
$$Wf=\cF (e^{\pi i z\cdot B_\cA z})\ast W_\cA f$$
	and the thesis follows arguing as in the previous part.
\end{proof}

\begin{proposition}
	Let $\cA\in Sp(2d,\bR)$ be a covariant matrix as in \eqref{A-covariant-L} with $B_\cA$ as in \eqref{aggiunta3-FL} and $B_\cA$ invertible (equivalently, $(1/2)I_{d\times d}-A_{11}$ invertible). Then,
	for $0< p,q\leq\infty$,  $f\in\mathcal{S}'\left(\mathbb{R}^{d}\right)$, we have 
	\[
	Wf\in \cF L^{p,q}_{{v_{s}\otimes 1}}\left(\rdd\right) \Leftrightarrow W_{\cA }f\in \cF L^{p,q}_{v_{s}\otimes 1}\left(\rdd\right), \quad s\in\bR.
	\]		
\end{proposition}
\begin{proof}
Taking the symplectic Fourier transfrom of both time-frequency representations:
	$$\cF_\sigma W_{\cA} f=\cF_\sigma\sigma_\cA \cdot Amb\left(f\right)$$  the claim is equivalent to showing  
	$$  \cF_\sigma W_{\cA}\in  L^{p,q}_{{v_{s}\otimes 1}} \Leftrightarrow  Amb\left(f\right) \in L^{p,q}_{{v_{s}\otimes 1}}.$$ 
Since both  $ \cF_\sigma \sigma_\cA (\zeta_1,\zeta_2)=e^{-\pi i z_1\cdot (\frac12 I_{d\times d}-A_{11})z_2}$ and $ (\cF_\sigma \sigma_\cA)^{-1} (\zeta_1,\zeta_2)=e^{\pi i z_1\cdot (\frac12 I_{d\times d}-A_{11})z_2}$ are in $L^\infty(\rdd)$, the statement follows by the point-wise product of mixed-norm spaces.
\end{proof}
\section{Schr\"{o}dinger equations with quadratic Hamiltonians}
Using the standard notation for the Cohen class (cf., e.g., \cite{grochenig}), for $\sigma\in \cS'(\rdd)$ we define the Cohen distribution $Q_{\sigma}$ by
\begin{equation}\label{Qsigma}
Q_{\sigma}f=\sigma\ast Wf,\quad f\in\cS(\rdd).
\end{equation}

\begin{proposition}\label{Wmetap}
	For $\chi\in Sp(d,\bR)$ we have 
	\begin{equation}
	Q_{\sigma}(\mu(\chi)f)(z)=Q_{\sigma_\chi}f(\chi^{-1}z),\quad z\in \rdd,
	\end{equation}
	with $\sigma_\chi(z)=\sigma(\chi z)$.
\end{proposition}
\begin{proof}
	From \cite[Proposition 1.3.7]{Elena-book} we have
	$$W(\mu(\chi)f)(z)=Wf(\chi^{-1}z),\quad f\in\cS(\rd),$$
	so that, for $\sigma\in\cS(\rdd)$, $f\in\cS(\rd)$,
	\begin{align*}
	Q_{\sigma}(\mu(\chi)f)(z)&= [\sigma\ast W(\mu(\chi)f)](z)=\intrdd W(\mu(\chi)f)(u)\sigma(z-u)du\\
	&= \intrdd Wf(\chi^{-1}u)\sigma(\chi (\chi^{-1}z-\chi^{-1}u))du\\
	&= \intrdd Wf(\zeta)\sigma(\chi (\chi^{-1}z-\zeta))d\zeta=W f\ast\sigma_{\chi}(\chi^{-1}z).
	\end{align*}
	For $\sigma\in\cS'(\rdd)$ one uses standard approximation arguments. This concludes the proof.
\end{proof}

We have now all the instruments to tackle the study of Schr\"{o}dinger equations. We consider the Cauchy problem in \eqref{C12} and express the solution as follows.
\begin{theorem}\label{Thrm-schrodinger}
	Let $u(t,\cdot)=e^{it\Opw(H)}u_0$, $t\in\bR$,  be the solution of the Cauchy problem in \eqref{C12}, with $Op_w(H)$ the Weyl quantization of the quadratic form $H$ in \eqref{I18}. If we set  $\chi_t=e^{t\mathbb{D}}\in Sp(d,\R)$, for $t\in\bR$, then \begin{equation}\label{App1}
	Q_\sigma(u(t,\cdot))(z)=Q_{\sigma_t}(u_0)(\chi_t^{-1}z),
	\end{equation}
	where 
	$$\sigma_t(z)=\sigma(\chi_t z).$$
\end{theorem}
\begin{proof}
	Observe that the solution can be written as  $u(t,\cdot)=e^{it\Opw(H)}u_0=\mu(\chi_t)$ where $\mu(\chi_t)$ is the continuous family of metaplectic operators with projections $\chi_t\in Sp(d,\bR)$ and $\chi_0=Id$ identity operator (cf. \cite{folland, Birkbis}). Using the covariance property for the Cohen class in Proposition \eqref{Wmetap}, we can write
	$$Q_\sigma(u(t,\cdot))(z)=Q_\sigma(\mu(\chi_t)u_0)(z)=Q_{\sigma_t}(u_0)(\chi_t^{-1}z),$$
	as desired.
\end{proof}
\begin{example}
	If  $\sigma=\delta$ we obtain
	$$W(u(t,\cdot))(z)=W u_0(\chi_t^{-1}z),$$
	as expected.
\end{example}

Let us limit to Cohen distributions generated by covariant matrices $\cA\in Sp(2d,\bR)$. Namely
\begin{equation}\label{QA}
Q_\sigma f =W_\cA f=W f\ast\sigma_\cA. 
\end{equation}
with kernel $\sigma_\cA$ in \eqref{aggiunta7}.

\begin{proposition}\label{simpl-cov}
Under the assumptions of Theorem \ref{Thrm-schrodinger} with a Cohen distribution $Q_\sigma$ as in \eqref{QA}, if we set  $\chi_t=e^{t\mathbb{D}}\in Sp(d,\R)$, for $t\in\bR$, then \begin{equation}\label{App2}
Q_\sigma(u(t,\cdot))(z)=W_{\cA}(u(t,\cdot))(z)= W_{\cA_t}u_0(\chi_t^{-1}z),
\end{equation}
where 
$W_{\cA_t}f(z)= W f\ast \sigma_{\cA_t}(z)$ and
$$\sigma_{\cA_t}(z)=\cF^{-1}\left(e^{-\pi i \zeta\cdot B_{\cA_t}\zeta }\right)(z),$$
and 
$$ B_{\cA_t}:= (\chi_t^{-1})^T B_\cA \chi_t^{-1}.$$
We have the equivalence of conditions for $0<p\leq 2$, $s\geq0$: 
\begin{itemize}
	\item [(i)] $u_0\in M^p_{v_s}(\rd)$
	\item[(ii)]$W_{\cA}(u(t,\cdot))\in M^p_{v_s}(\rdd)$
	\item[(iii)] $W_{\cA_t}u_0 \in M^p_{v_s}(\rdd)$.
\end{itemize}
\end{proposition}
\begin{proof}
We use the dilation properties of the \ft. In fact, $$\cF^{-1}\left(e^{-\pi i \zeta\cdot B_{\cA_t}\zeta }\right)(\chi_t z)=\cF^{-1}\left(e^{-\pi i \chi_t^{-1}\zeta\cdot B_{\cA_t}\chi_t^{-1}\zeta }\right)( z)=\cF^{-1}\left(e^{-\pi i \zeta\cdot (\chi_t^{-1})^T B_{\cA_t}\chi_t^{-1}\zeta }\right)( z)$$
(recall that  $\det \chi_t=1$). The equivalence of (i), (ii) and (iii) follows from Theorem \ref{Teorema2}.
\end{proof}

\begin{proposition}\label{Prop4.5}
	Under the hypotheses  of Proposition \ref{simpl-cov}, if  we assume  $\cA$ shift-invertible  then $\cA_t$ is shift-invertible. We have the equivalence of conditions for $1\leq p\leq 2$, $s\geq0$: 
	\begin{itemize}
		\item [(i)] $u_0\in M^p_{v_s}(\rd)$
		\item[(ii)]$W_{\cA}(u(t,\cdot))\in L^p_{v_s}(\rdd)$
		\item[(iii)] $W_{\cA_t}u_0 \in L^p_{v_s}(\rdd)$.
	\end{itemize}
\end{proposition}
\begin{proof}
For every $t\in\bR$, the relation between $B_{\cA_t}$ and $ E_{\cA_t}$ is given by \eqref{E-B-eq}, so that 
$$ E_{\cA_t}=B_{\cA_t} J^{-1} -\frac12 I_{d\times d}.$$
 Since $B_{\cA_t}=(\chi^{-1})^TB_\cA \chi^{-1}_t,$ we can view the matrix $ E_{\cA_t}$ in terms of the matrix $E_\cA$ as follows: 
\begin{align*}
E_{\cA_t}&=B_{\cA_t}J^{-1}-\frac12 I_{d\times d}\\
&=(\chi_t^{-1})^TB_\cA \chi^{-1}_t J^{-1}-\frac12 I_{d\times d}\\
%&=(\chi_t^{-1})^T\left(E_\cA +\frac12 I_{d\times d} \right)((\chi^{-1}_t)^T)^{-1}J J^{-1}-\frac12 I_{d\times d}\\
&=(\chi_t^{-1})^T\left(E_\cA +\frac12 I_{d\times d} \right)J\chi^{-1}_t J^{-1}-\frac12 I_{d\times d}\\
&=(\chi_t^{-1})^T\left(E_\cA +\frac12 I_{d\times d} \right)((\chi^{-1}_t)^T)^{-1}J J^{-1}-\frac12 I_{d\times d}\\
&=(\chi_t^{-1})^T E_\cA ((\chi^{-1}_t)^T)^{-1}+\frac12 I_{d\times d}-\frac12 I_{d\times d}\\
&=(\chi_t^{-1})^T E_\cA ((\chi^{-1}_t)^T)^{-1}.
\end{align*}
Since $(\chi_t^{-1})^T$ is invertible, $E_{\cA_t}$ is invertible if and only if $E_\cA$ is.  The equivalence of (i), (ii) and (iii) follows from Corollary \ref{mainEcor}.
\end{proof}

Observe that the previous result does not require the assumption \eqref{metapf2dil}.

\textbf{Example: The free particle.}
Consider the Cauchy problem for the
Schr\"odinger equation
\begin{equation}\label{cp}
\begin{cases}
i\partial_t u+\Delta u=0\\
u(0,x)=u_0(x),
\end{cases}
\end{equation}
with $x\in\R^d$, $d\geq1$. The  explicit formula for the solution $u(t,x)= e^{i t \Delta}u_0(x)$ is
\begin{equation}\label{sol}
u(t,x)=(K_t\ast u_0)(x),
\end{equation}
where
\begin{equation}\label{chirp0}
K_t(x)=\frac{1}{(4\pi i t)^{d/2}}e^{i|x|^2/(4t)}.
\end{equation}
The canonical transformation $\chi_t$ is given by
\begin{equation}\label{part-lib}
\chi_t(y,\eta)=(y+4\pi t \eta,\eta)=\left(\begin{array}{cc}
I_{d\times d}&(4\pi t ) I_{d\times d}\\
0_{d\times d} &  I_{d\times d}
\end{array}\right)\left(\begin{array}{c}
y\\
\eta
\end{array}\right), 
\end{equation}
so that 
$$\chi_t^{-1}=\left(\begin{array}{cc}
I_{d\times d}& - (4\pi t ) I_{d\times d}\\
0_{d\times d} &  I_{d\times d}
\end{array}\right).
$$
We may apply Proposition \ref{simpl-cov} with $B_{\cA_t}$ and $\cA_t$ defined consequently. Assuming further shift-invertibility, we may apply Proposition \ref{Prop4.5} as well.  It is clear, in this context, that starting  with a symplectic matrix $\cA$ of the type \eqref{metapf2dil} does not guarantee that the new matrix  $\cA_t$ in Proposition \ref{simpl-cov}  satisfies condition \eqref{metapf2dil}.
In fact, applying \eqref{part-lib} to the matrix $B_\cA$ in \eqref{aggiunta3-FL}, we obtain 
$$B_{\cA_t}=(\chi_t^{-1})^TB_\cA \chi_t^{-1}=\left(\begin{array}{cc}
0_{d\times d}& \frac12 I_{d\times d}-A_{11}\\
\frac12 I_{d\times d}-A_{11}^T &  (4\pi t)(A_{11}+A_{11}^T-I_{d\times d})
\end{array}\right).$$
The matrix $B_{\cA_t}$  is  of the type \eqref{aggiunta3-FL} if and only if   
\begin{equation}\label{conserv-A}
A_{11}+A_{11}^T=I_{d\times d},
\end{equation}
hence if the previous condition is not fulfilled  $\cA_t$ is not of the type \eqref{metapf2dil}.  

We test condition \eqref{conserv-A} on the 
the $\tau$-Wigner representations, for any $\tau\in\bR$ and with ${\bf \cA }_\tau$ defined in \eqref{Aw-tau}. In this case  $A_{11}+A_{11}^T=2(1-\tau)I_{d\times d}$ and we obtain condition \eqref{conserv-A} if and only if $\tau=1/2$ (the expected Wigner case).
By a direct computation: 
\begin{equation}\label{FP1}
W_\tau u(t,x,\xi)=W_{\tau,t} u_0(x-4\pi t \xi,\xi),
\end{equation}
where the representation $W_{\tau,t}$ is of Cohen class:
\begin{equation}\label{C28}
W_{\tau,t}f=Wf\ast\sigma_{\tau,t},
\end{equation}
with
\begin{equation}\label{C29}
\sigma_{\tau,t}(y,\eta)=\sigma_\tau(\chi_t(y,\eta))=\sigma_\tau(y+4\pi t \eta, \eta),
\end{equation}
and
\begin{equation*}
\sigma_\tau \phas= 
\begin{cases}
\frac{2^d}{|2\tau -1|^d}e^{2\pi i\frac{2}{2\tau -1}x \xi} &\tau \neq \frac{1}{2}\\
&\\
\delta & \tau =\frac{1}{2},
\end{cases}
\end{equation*}
cf.  Proposition 1.3.27 in \cite{Elena-book}.\par
 We may write $W_{\tau,t}$ in the form of an $\cA_t$-Wigner representation, with 
\begin{equation}\label{C30}
\mu(\cA_t) F(x,\xi)= \intrd e^{-2\pi i (y\xi+2\pi t(1-2\tau)y^2)}F(x+\tau y,x-(1-\tau)y)\,dy.
\end{equation}
\begin{definition}\label{4.6}
	For $\cA\in Sp(2d,\bR)$, $f\in\cS'(\rd)$, $0<p<\infty$, $s\geq0$, we say that $z_0=(x_0,\xi_0)\notin\mathcal{W}\cF^{p,s}_\cA(f) $, $z_0\not=0$, if there exists $\Gamma_0$, conic neighbourhood of $z_0$, such that 
	\begin{equation}\label{26Aripet}
	\int_{\Gamma_{z_0}} \la z\ra^{ps} |W_{\cA}f(z)|^p\,dz<\infty.
	\end{equation}
\end{definition}
The wave front set $\mathcal{W}\cF^{p,s}_\cA(f) $ is a closed cone in $\rdd\setminus\{0\}$.

In our context, it will be convenient to limit the definition to shift-invertible matrices $\cA$ and $1\leq p\leq2$. 
\begin{proposition}\label{4.7}
	In the preceding Definition \ref{4.6} assume $f\in L^p(\rd)$, $1\leq p\leq 2$, $s\geq0$ and let $\cA$ be shift-invertible. Then $\mathcal{W}\cF^{p,s}_\cA(f)=\emptyset$ if and only if $f\in M^p_{v_s}(\rd)$.
\end{proposition}
\begin{proof}
	Under such assumptions, from Corollary \ref{mainEcor} we have that $f\in M^p_{v_s}(\rd)$ if and oly if $W_\cA f \in L^p_{v_s}(\rdd)$. So, if $f\in M^p_{v_s}(\rd)$ then \eqref{26Aripet} is satisfied in every cone $\Gamma_{z_0}$, for all $z_0\not=0$, hence $ \mathcal{W}\cF^{p,s}_\cA(f)=\emptyset$. In the opposite direction, assume  $ \mathcal{W}\cF^{p,s}_\cA(f)=\emptyset$, that is \eqref{26Aripet} is satisfied for a suitable conic neighbourhood $\Gamma_{z_0}$ of any $z_0\not=0$.  From the compactness of the sphere $\mathbb{S}^{2d-1}$ we deduce that the integral \eqref{26Aripet} is convergent over the whole $\rdd$, i.e., $W_\cA f \in L^2_{v_s}(\rdd)$. This completes the proof. 
\end{proof}

Assuming further that $\cA$ is covariant, we consider the Schr\"{o}dinger equation \eqref{C12} and define the covariant matrix $\cA_t$, $t\in\bR$, as in Proposition \ref{simpl-cov}. From Proposition \ref{Prop4.5} we have that, if $\cA$ is shift-invertible, so is $\cA_t$, for all $t\in\bR$.
\begin{theorem}\label{4.8}
Assume $u_0\in\lrd$. Let $u(t,\cdot)\in \lrd$, $t\in\bR$,  be the solution of \eqref{C12}. Let $\cA$ be covariant and shift-invertible. Then, for $1\leq p\leq 2$, $s\geq0$:
\begin{equation}\label{26Brip}
\mathcal{W}\cF^{p,s}_{\cA}(u(t,\cdot))=\chi_t(\mathcal{W}\cF^{p,s}_{\cA_t}(u_0)). 
\end{equation}
\end{theorem} 
\begin{proof}
	Assume $\zeta_0\not=\mathcal{W}\cF^{p,s}_{\cA}(u_0)$, i.e., there exists $\Lambda_{\zeta_0}$, conic neighbourhood of $\zeta_0$, such that
	\begin{equation}\label{(111)}
	\int_{\Lambda_{\zeta_0}} \la \zeta\ra^{ps} |W_{\cA_t} (u_0)(\zeta)|^pd\zeta<\infty.
	\end{equation}
\end{proof}

Observe that $\Gamma_{z_0}=\chi_t^{-1}(\Lambda_{\zeta_0})$ is a conic neighbourhood of $z_0$. We have, by applying \eqref{App2} and setting $z=\chi_t(\zeta)$:
\begin{align*}
\int_{\Gamma_{z_0}} \la z\ra^{ps} |W_\cA(u(t,\cdot))(z)|^p dz &= \int_{\Gamma_{z_0}} \la z\ra^{ps} |W_{\cA_t}(u_0)(\chi_t^{-1}z)|^p dz\\
&=\int_{\Lambda_{\zeta_0}} \la \chi_t \zeta\ra^{ps} |W_{\cA_t}(u_0)(\zeta)|^p  d\zeta<\infty,
\end{align*}
since 
$\la \chi_t \zeta\ra^{ps}\asymp \la \zeta\ra^{ps}$, and we can apply \eqref{(111)}. Hence $z_0=\chi_{t}\zeta_0\notin  \mathcal{W}\cF^{p,s}_{\cA}(u(t,\cdot))$. Arguing similarly in the opposite direction, we obtain \eqref{26Brip}.

\section*{Acknowledgements}
The authors have been supported by the Gruppo Nazionale per l’Analisi Matematica, la Probabilità e le loro Applicazioni (GNAMPA) of the Istituto Nazionale di Alta Matematica (INdAM).

\end{document}